\documentclass[3p,12pt]{elsarticle}

\usepackage{amsthm}
\usepackage{amssymb}
\usepackage{amsmath}
\usepackage{mathtools}
\usepackage{mathrsfs}
\usepackage{stmaryrd}
\usepackage{dsfont}

\usepackage{tikz}
\usetikzlibrary{cd,arrows,babel}
\usepackage{tikz-cd}

\usepackage[english]{babel}

\usepackage{array,colortbl}
\usepackage{etoolbox}
\usepackage{hyperref}

\usepackage{lineno}

\newcommand{\N}{\mathbb{N}}
\newcommand{\R}{\mathbb{R}}

\newcommand{\K}{\mathbb{K}}

\newcommand{\I}{\mathcal{I}}
\newcommand{\J}{\mathcal{J}}
\newcommand{\fin}{\mathsf{Fin}}

\DeclareMathOperator{\ran}{ran}
\DeclareMathOperator{\spn}{span}
\DeclareMathOperator{\cl}{cl}
\DeclareMathOperator{\ba}{ba}

\DeclareMathOperator{\bor}{Bor}
\DeclareMathOperator{\supp}{supp}
\DeclareMathOperator{\Stone}{Stone}
\DeclareMathOperator{\dist}{dist}

\DeclareMathOperator{\Dil}{Dil}
\DeclareMathOperator{\dens}{dens}

\newcommand{\cardc}{\mathfrak{c}}

\newcommand{\Balg}{\mathfrak{B}}

\theoremstyle{plain}

\newtheorem{theorem}{Theorem}[section]
\newtheorem{lemma}[theorem]{Lemma}
\newtheorem{proposition}[theorem]{Proposition}
\newtheorem{corollary}[theorem]{Corollary}

\newtheorem{claim}{Claim}
\AtEndEnvironment{claim}{\null\hfill$\lozenge$}
\AtEndEnvironment{proof}{\setcounter{claim}{0}}

\theoremstyle{definition}

\newtheorem{definition}[theorem]{Definition}
\newtheorem{example}[theorem]{Example}
\newtheorem{question}[theorem]{Question}

\theoremstyle{remark}

\newtheorem{remark}[theorem]{Remark}

\journal{Journal of Functional Analysis}

\begin{document}

\begin{frontmatter}



\title{Structure properties of Banach spaces of $\mathcal I$-null sequences}


\author[uis]{Michael Alexánder Rincón-Villamizar}
\ead{marinvil@uis.edu.co}
\author[unesp]{Victor dos Santos Ronchim\corref{cor1}}
\cortext[cor1]{Corresponding author}
\ead{victor.ronchim@unesp.br}
\author[uis]{Carlos Uzcategui Aylwin}
\ead{cuzcatea@saber.uis.edu.co}

\affiliation[uis]{organization={Universidad Industrial de Santander},
            addressline={Carrera 27 Calle 9}, 
            city={Bucaramanga},
            postcode={680002}, 
            state={Santander},
            country={Colombia}}
\affiliation[unesp]{organization={Universidade Estadual Paulista},
            addressline={Avenida Universitária, 3780}, 
            city={Botucatu},
            postcode={18610-034}, 
            state={São Paulo},
            country={Brazil}}

\begin{abstract}
We investigate structural properties of the ideal-null sequence spaces $c_{0,\mathcal I}$ and $c_{0,\mathcal I}(X)$, as well as of the quotient $\ell_\infty/c_{0,\mathcal I}$, with emphasis on the interplay between Banach space theory and combinatorial, topological, and measure-theoretic properties of the underlying ideal. In this setting, we address several classical problems concerning $c_0$, including the description of the dual space, compactness criteria, bounded and compact operators with range in $c_{0,\mathcal I}$, Sobczyk-type extension phenomena, and complemented copies of $c_0$. We identify $c_{0,\mathcal I}^*$ isometrically with a space of bounded finitely  additive measures on $\mathcal I$ and obtain a canonical atomic--singular decomposition of the dual, with $\ell_1$ as an isometrically complemented summand.
A Dini principle for ideal convergence yields characterizations of relatively compact subsets of $c_{0,\mathcal I}$ and, in turn, of compact operators with range in $c_{0,\mathcal I}$. We also establish a Sobczyk theorem for $c_{0,\mathcal I}$ and show that its separable-injectivity constant is exactly $2$ for every proper ideal. Concerning complementation, we construct an explicit ZFC class of statistical ideals for which $c_{0,\mathcal I}$ is not complemented in $\ell_\infty$, and we characterize the existence of complemented copies of $c_0$ in $c_{0,\mathcal I}(X)$ in terms of the corresponding property for $c_{0,\mathcal I}$ and $X$. 
Finally, we study the structure of $\ell_\infty/c_{0,\mathcal I}$. 
We obtain isometric Banach-lattice descriptions of quotients associated with direct and Frol\'ik sums of ideals and characterize Banach-lattice copies of $c_0(\kappa)$ in $\ell_\infty/c_{0,\mathcal I}$ in terms of $\mathcal I$-almost disjoint families. As consequences, we identify $\mathfrak{ad}(\mathcal I)$ with the cellularity of the Stone space associated with $\mathcal I$ and determine the behavior of this cardinal invariant under Fubini products.
\end{abstract}

\begin{keyword}
$\mathcal I$-null sequences \sep
statistical ideals \sep
Banach sequence spaces \sep
complemented copies of $c_0$ \sep
Sobczyk theorem \sep
finitely additive measures \sep
$\mathcal{I}$-almost disjoint families

\MSC[2020] 46B45 \sep 03E15 \sep 46B20 \sep 46B26 \sep 28A33
\end{keyword}
\end{frontmatter}


\newpage

\tableofcontents

 \section{Introduction}
  
 An ideal on $\mathbb{N}$ is a collection $\I$  of subsets of $\mathbb{N}$ closed under finite unions and taking subsets of its elements. A sequence $(x_n)$ in a Banach space $X$ is said to be $\mathcal{I}$-convergent to $x\in X$, denoted as $\mathcal{I}$-$\lim x_n=x$ or $x_n\stackrel{\I}{\to}x$, if for each $\varepsilon>0$, $A((x_n),\varepsilon)\doteq\{n\in\mathbb{N}\,\colon\,\|x_n-x\|\geq\varepsilon\}\in\mathcal{I}$. 
When $\mathcal{I} =\fin$, the ideal of finite subsets of $\mathbb{N}$, we have the classical convergence in $X$. For this reason, it is natural—and we will adopt this assumption—to require that $\fin$ is contained in every ideal under consideration.  The notion of $\mathcal{I}$-convergence was introduced in \cite{k-s-w}, although many authors had already studied this concept in particular cases and in different contexts (see, for instance, \cite{balcerzak,b-s-w,fast,filipow,k-m-s,Leonetti2018}). In this paper, we study the following space:
\[c_{0,\mathcal{I}}(X)=\{(x_n)\in X^{\mathbb N}\,\colon\,(x_n)\text{ is bounded and }\mathcal{I}-\lim x_n=0\}.\]
If $X$ is the scalar field, we simply write $c_{0,\mathcal I}$ instead of $c_{0,\mathcal I}(X)$. It is easy to see that $c_{0,\mathcal I}$ is a closed subspace of $\ell_\infty$.

As shown in \cite{rincon-uzcategui}, the Banach and Banach-lattice properties of $c_{0,\mathcal I}$ are closely related to the combinatorial and topological properties of the ideal $\mathcal{I}$. We will continue the work done in \cite{rincon-uzcategui} by contributing to the development of a structural theory for these spaces. 

For Banach space theorists, the study of structural properties of classical sequence  spaces is of particular interest. 
In the case of $c_0$,  among the fundamental questions that were well studied in the literature, we highlight: the description of $c_0^*$, the structure of its compact subsets, extension of $c_0$-valued operators, compact operators into $c_0$ and complemented copies of $c_0$ in $\ell_\infty$-sums. One of the purposes of this paper is to investigate these topics for the Banach spaces $c_{0,\mathcal I}$ and $c_{0,\mathcal I}(X)$.

A second theme of the paper concerns the quotient space: 
\[\ell_\infty/c_{0,\mathcal{I}}=
\{{\bf x}+c_{0,\mathcal I}\,\colon\,{\bf x}\in\ell_\infty\}.\] 
Just as $c_{0,\mathcal I}$ reflects properties of the ideal $\mathcal I$, the quotient $\ell_\infty/c_{0,\mathcal I}$ also encodes important combinatorial and topological information about $\mathcal I$. 
In particular, its Banach-lattice structure turns out to be closely related to the existence of $\mathcal I$-almost disjoint families.

The remainder of the paper is organized as follows. After the preliminary material given in Section \ref{preliminares}, Section~\ref{section dual of c0I} is devoted to the dual structure of $c_{0,\mathcal I}$. We identify $c_{0,\mathcal I}^*$ isometrically with the space $\ba(\mathcal I)$ of bounded finitely additive measures on $\mathcal I$ (Theorem~\ref{thm:dual-ba}) and obtain an isometric embedding of $c_{0,\mathcal I}(X)^*$ into $\mathrm{ba}(\mathcal I,\ell_\infty(X)^*)$ in the vector-valued setting (Theorem~\ref{thm: dual c0I(X)}). 
These results parallel the classical description of $\ell_\infty^*$ (see \cite[Section 10.10 and Theorem 14.4]{Aliprantis2006}). We then show that the classical sequence-space approaches to duality do not produce new spaces in this context and obtain an isometric decomposition of $c_{0,\I}^*$ in which $\ell_1$ appears as a complemented subspace (Theorem~\ref{thm:dual-splitting}).

Section~\ref{section compactness} turns to compact subsets of $c_{0,\mathcal I}$. The starting point is the uniform ideal-tail principle (Theorem~\ref{thm:ideal-dini}), which may be regarded as an ideal version of the uniform conclusion in Dini's theorem. As a consequence, Theorem~\ref{thm:compactness-c0I} provides different characterizations of relative compactness in $c_{0,\mathcal I}$ and extends the classical characterization for $c_0$, which is recovered as the particular case where $\mathcal I=\fin$ (Proposition~\ref{prop: rel compact c0}). Moreover, under the representation of $c_{0,\mathcal I}$ as a $C_0(L)$-space, these equivalent conditions provide a concrete counterpart of the Arzel\`a--Ascoli theorem adapted to the geometry determined by $\mathcal I$. 

The compactness theory developed in Section~\ref{section compactness} feeds directly into Section~\ref{section sobzcyk}, where we study $c_{0,\mathcal I}$-valued operators. 
We first identify $B(X,c_{0,\mathcal I})$ isometrically with the space of bounded sequences in $X^*$ that are pointwise $\mathcal I$-null, in analogy with the classical description of $c_0$-valued operators. 
Combining Schauder's theorem with the uniform ideal-tail principle, we then obtain a characterization of compact operators $T:X\to c_{0,\mathcal I}$ in terms of the norm behavior and relative norm compactness of their coordinate functionals (Theorem~\ref{thrm:compact-operators}). 
In the second part of the section, we establish an ideal version of Sobczyk's theorem (Theorem~\ref{thm:sobczyk}) by a classical argument. 
Although this result can also be obtained from the theory of separably injective Banach spaces, which has been developed extensively by several authors (see \cite{aviles-et-all} for an excellent exposition of that theory), it is natural to seek a classical proof in the present setting. 
For every proper ideal $\mathcal I$, the separable-injectivity constant of $c_{0,\mathcal I}$ is exactly $2$ (Theorem~\ref{thm:sharp-sobczyk}). Corollary~\ref{cor:no-below-two} gives the same lower bound for the norm of any projection from $\ell_\infty$ onto $c_{0,\mathcal I}$.

Section~\ref{section complementation} focuses on the complementation of $c_{0,\mathcal I}$ in $\ell_\infty$. 
In the terminology of \cite{rincon-uzcateguiII}, we call a proper ideal $\mathcal I$ \textit{complemented} if $c_{0,\mathcal I}$ is complemented in $\ell_\infty$. 
Recently, Hrušák and Sáenz \cite{HrusakSaenz} obtained a characterization of complemented ideals in terms of the quotient $\ell_\infty/c_{0,\mathcal I}$, namely, $\mathcal I$ is complemented if and only if $\ell_\infty/c_{0,\mathcal I}$ is isomorphic to a closed subspace of $\ell_\infty$. Using this characterization, we first observe that every complemented ideal is a statistical ideal, that is, the null ideal determined by a finitely additive probability measure that vanishes on $\fin$ (the so-called statistical measures in \cite{kadetsetal2022}). We will show that the converse fails by exhibiting in Corollary~\ref{cor:lacunary-U} an explicit ZFC class of non-complemented statistical ideals with a simpler description than the example given in \cite{HrusakSaenz}.

In Section~\ref{section complemented copies of c0} we address a different complementation problem, namely, when the vector-valued space $c_{0,\mathcal I}(X)$ contains a complemented copy of $c_0$. Extending the classical characterization of complemented copies of $c_0$ in $\ell_\infty(X)$ due to \cite{leung-rabiger}, we obtain in Theorem~\ref{thm-grothendieck} a complete characterization: $c_0$ is complemented in $c_{0,\mathcal I}(X)$ if and only if it is complemented either in $c_{0,\mathcal I}$ or in $X$. The proof makes essential use of finitely additive measures, thus returning to the measure-theoretic viewpoint developed in Section~\ref{section dual of c0I}.

Finally, Section~\ref{section quotiens} studies the quotient $\ell_\infty/c_{0,\mathcal I}$ and shows how its Banach-lattice structure reflects combinatorial properties of $\mathcal I$. 
The first part of the section provides isometric Banach-lattice descriptions of quotients associated with direct and Frol\'ik sums of ideals, including the representation of $\ell_\infty(X)/c_{0,\mathcal I\times\mathcal J}(X)$ obtained in Theorem~\ref{isomorphism quotient of a fubini-product}. 
In the second part, Theorem~\ref{AD-families and KI} shows that, for
every infinite cardinal $\kappa$, the quotient $\ell_\infty/c_{0,\mathcal I}$
contains a Banach-lattice copy of $c_0(\kappa)$ if and only if there
is an $\mathcal I$-almost disjoint family of size $\kappa$. It also
identifies $\mathfrak{ad}(\mathcal I)$, the supremum of the sizes of
all such families \cite{rincon-uzcateguiII}, with the common cellularity
of the quotient Boolean algebra $\mathcal P(\mathbb N)/\mathcal I$
and its Stone space $K_{\mathcal I}$. 
For proper admissible ideals $\mathcal I$ and $\mathcal J$ on $\mathbb N$,
Theorem~8.11 shows that $\mathfrak{ad}(\mathcal I\times\mathcal J)$
equals $k\cdot\mathfrak{ad}(\mathcal I)$ when
$\mathfrak{ad}(\mathcal J)=k<\aleph_0$ and $2^{\aleph_0}$ otherwise.
When either factor is $\mathrm{Fin}$, this recovers the computation
in \cite[Theorem~7.10(1)]{rincon-uzcateguiII} and Example~\ref{exa: id changing factors} shows that interchanging the factors can change the value from $2^{\aleph_0}$
to $\aleph_0$.
\section{Preliminaries}
\label{preliminares}

\subsection{Ideals} 

An ideal $\mathcal I$ on  a set $X$ is a collection of subsets of $X$ satisfying:
\begin{enumerate}[(1)]
    \item $\emptyset\in\mathcal I$;
    \item If $A\subseteq B$ and $B\in\mathcal I$, then $A\in\mathcal I$;
    \item If $A,B\in\mathcal I$, then $A\cup B\in\mathcal I$.
\end{enumerate}

We always assume that every finite subset of $X$ belongs to $\mathcal{I}$. The dual filter of an ideal $\mathcal{I}$ is denoted by $\mathcal{I}^*$ and consists of all sets of the form $X\setminus A$ for some $A\in \mathcal{I}$. The {\em co-ideal} $\mathcal{I}^+$ is the collection $\mathcal{P}(X)\setminus \mathcal{I}$.  

An ideal $\I$ is {\em maximal} if $\mathcal P(X)$ is the only ideal properly  extending  $\I$; equivalently, if $\I^*$ is an ultrafilter. Notice that $\I$ is maximal if $\I^*=\I^+$. For $A\subseteq X$, we denote the restriction of $\I$ to $A$ by $\mathcal I\restriction A=\{A\cap B\,\colon B\in\mathcal I\}$ which is an ideal on $A$. Let  $\mathcal A$ and $\mathcal B$ be families of sets, we denote by $\mathcal A\sqcup\mathcal B$ the collection $\{A\cup B\,\colon\,A\in\mathcal A,B\in\mathcal B\}.$
An $\mathcal I$-AD family is a collection $\mathcal{A}\subseteq \I^+$ such that $A\cap B\in \I$ for every two different sets $A, B\in \mathcal{A}$.  Two ideals $\I$ and $\J$ on $X$ and $Y$, respectively, are {\em isomorphic} if there is a bijection $f:X\to Y$ such that $E\in \I$ iff $f[E]\in \J$ for all $E\subseteq X$. 

Let $\{K_n:\; n\in F\}$ be a partition of a countable set $X$, where $F\subseteq \N$. For $n\in F$, let $\I_n$ be an ideal on $K_n$. The direct sum, denoted by $\bigoplus\limits_{n\in F} \I_n$, is defined as follows:
\[
A\in \bigoplus_{n\in F} \I_n \Leftrightarrow (\forall n\in F)(A\cap K_n \in \I_n).
\]
Notice that the direct sum $\bigoplus_n \I_n$ can also be  naturally  defined in $\N\times\N$.  When $\mathcal I_n$ is isomorphic to $\mathcal I$  for all $n$, the direct sum is denoted by $\mathcal I^\omega$.

If $\mathcal I$ is an ideal on $\mathbb N$ and $(\mathcal I_n)_{n\in\mathbb N}$ is a sequence of ideals on $\mathbb N$, the \textit{Frol\'ik sum}
of $(\mathcal I_n)$ by $\mathcal I$ is defined by 
\begin{equation*}
    \sum_{\mathcal I}\mathcal I_n=\bigg\{A\subseteq\mathbb N\times\mathbb N\,\colon\,\{m\in\N:A_{(m)}\not\in\mathcal I_m\}\in\mathcal I\bigg\},
\end{equation*}
where $A_{(m)}=\{n\in\mathbb N\,\colon\,(m,n)\in A\}$. In particular, when $\mathcal I_m=\mathcal J$ for each $m\in\mathbb N$, the Frol\'ik sum coincides with the \textit{Fubini product} of $\mathcal I$ and $\mathcal J$, denoted by $\mathcal I\times\mathcal J$.

\medskip

Recall that $\beta\omega$ is the Stone-\v{C}ech compactification of $\omega$ which  is usually identified with the collection of all ultrafilters on $\omega$. For a set $A\subseteq\omega$, we let $\widehat{A}=\{p\in\beta\omega\,\colon\,A\in p\}$.  The family $\{\widehat{A}\,\colon\,A\subseteq\omega\}$ is    a basis for the topology of $\beta\omega$. As usual, we identify each $n\in \omega$ with the principal ultrafilter $\{A\subseteq\omega: n\in A\}$.  Every principal ultrafilter is an isolated point of $\beta\omega$. We set $U_\mathcal I=\{p\in\beta\omega\,\colon\,\I\cap  p\neq \emptyset\}$ and
$K_\mathcal I=\{p\in\beta\omega:\I^*\subseteq p\}=\{p\in\beta\omega:\I\cap p=\emptyset\}$. Since $\displaystyle U_{\mathcal I}=\bigcup\{\widehat{A}\,\colon\,A\in\mathcal I\}$ and $K_\mathcal I=\beta\omega\setminus U_\mathcal I$,  $U_{\mathcal I}$ is open in $\beta\omega$ and $K_\mathcal I$ is compact. Under the previous identification, $\omega\subseteq U_\I$ for every ideal $\I$. 

  If $\mathcal I$ is an ideal on $\omega$, then  $[A]_{\mathcal I}$ denotes the equivalence class of $A$ in the Boolean algebra
  $\mathcal P(\omega)/\mathcal I$; thus  $[A]_{\mathcal I}=[B]_{\mathcal I}$ precisely when  $A\mathbin\triangle B\in\mathcal I$. By $\Stone(\mathbb B)$ we denote the Stone space of the Boolean algebra $\mathbb B$.

\subsection{Banach spaces} 

We will use standard terminology and notation for Banach lattices and Banach space theory. For unexplained definitions and notations, we refer to \cite{AK,schaeffer}. The scalar field is denoted by $\mathbb K$. All Banach lattices analyzed here are assumed to be real. However, our results can be extended to complex Banach lattices in the usual manner \cite[Chapter 2, p. 133]{schaeffer}. A subset $X$ of a Banach lattice $E$ is \textit{solid} if for all $x\in E$ and $y\in X$ such that $|x|\leq |y|$, we have $x\in X$. An ideal of $E$ is a closed solid sublattice. As shown in \cite[Proposition 4.1]{rincon-uzcategui}, every ideal of $\ell_\infty$ containing $c_0$ has the form $c_{0,\mathcal I}$ for some ideal $\I$. 

If $X$ and $Y$ are isomorphic (isometric) Banach spaces, we write $X\sim Y$ ($X\equiv Y$, respectively). We also write $X\equiv_{\leq}Y$ whenever $X$ and $Y$ isometric as Banach lattices. If $E$ is a closed subspace of a Banach space $X$, we say that 
$E$ is complemented in $X$ if there is a continuous onto operator $P\colon X\to E$ such that $P^2=P$, or equivalently, there is a closed subspace $W$ of $X$ such that $X=E\oplus W$. In addition, if $E$ and $X$ are Banach lattices, $P$ is called \textit{positive} if $Px\geq{\bf0}$ for all $x\geq{\bf0}$.  We write $X\xhookrightarrow{c} Y$ whenever $Y$ contains a complemented subspace isomorphic to $X$.

Now, we introduce some notation. Let $(X_j)_{j\in\mathbb N}$ be a family of Banach spaces.
\begin{enumerate}[(1)]
    \item $\ell_\infty((X_j)_{j\in\mathbb N})$ denotes the $\ell_\infty$-sum of $(X_j)_{j\in\mathbb N}$, that is, the Banach space of all bounded sequences $\displaystyle(x_j)\in\prod_jX_j$  endowed with the norm $\|\cdot\|$
given by $\displaystyle\|(x_j)\|=\sup_j\|x_j\|$. 

\item $c_0((X_j)_{j\in\mathbb N})$ is the $c_0$-sum of $(X_j)_{j\in\mathbb N}$, that is, the closed subspace of $\ell_\infty((X_j)_{j\in\mathbb N})$ consisting of all null sequences, i.e., $\displaystyle\lim_j\|x_j\|=0$.
\end{enumerate}

When $X=X_j$ for all $j\in\mathbb N$, these spaces are denoted by  
$\ell_\infty(X)$ and $c_0(X)$, respectively. Finally, for $X=\mathbb K$, $\ell_\infty(X)$ and $c_0(X)$ correspond to $\ell_\infty$ and $c_0$, respectively.

\begin{remark}\label{isomorphisms of l_oo(c_0)}
    It is easy to see that if each $X_j$ is isometric (isomorphic) to $X$, then
$\ell_\infty((X_j)_{j\in\mathbb N})$ ($c_0((X_j)_{j\in\mathbb N})$) is isometric (isomorphic, respectively) to $\ell_\infty(X)$ ($c_0(X)$, respectively).
\end{remark}

 The following result gives a representation of $c_{0,\mathcal I}$ and $\ell_\infty/c_{0,\mathcal I}$ as spaces of continuous functions. 

 \begin{proposition}(\cite[Proposition 3.1]{rincon-uzcategui} and \cite{kania}, respectively)\label{isometrias}
     Let $\mathcal I$ be an ideal on $\mathbb N$.
     \begin{enumerate}[(1)]
         \item $c_{0,\mathcal I}$ is Banach isometric to $C_0(U_\mathcal I)$.
         \item $\ell_\infty/c_{0,\mathcal I}$ is Banach isometric to $C(K_\mathcal I)$.
     \end{enumerate}
 \end{proposition}

\begin{remark}\cite[Proposition 2.5]{rincon-uzcategui}\label{denseness}
   If $c_{00,\mathcal I} \doteq \spn \{\chi_A\,\colon\,A\in\mathcal I\} \subseteq \ell_\infty$, then  $\overline{c_{00,\mathcal I}}^{\|\cdot\|_\infty}=c_{0,\mathcal I}$.
\end{remark}

\section{The dual of $c_{0,\mathcal I}$}
\label{section dual of c0I}

In this section, we study the dual spaces associated with $\mathcal I$-null sequences through finitely additive measures on $\mathcal I$. We first identify $c_{0,\mathcal I}^*$ isometrically with a space of scalar-valued measures and then obtain a related isometric embedding in the vector-valued setting.

To treat the scalar and vector-valued settings within the same framework, we present the following:

\begin{definition}
Let $E$ be a normed space over $\mathbb K$. We denote by $\ba(\mathcal I,E)$ the space of all finitely additive maps $m\colon\mathcal I\to E$ of bounded variation; that is,
\[
\ba(\mathcal I,E)
\doteq
\left\{
m\colon\mathcal I\to E\,\colon\, m(A\cup B)=m(A)+m(B)\text{ whenever }A\cap B=\emptyset,\,\, \|m\|<\infty \right\},
\]
where $\|m\|$ is the total variation norm, defined by:
\[
\|m\| \doteq
\sup\left\{
\sum_{j=1}^n\|m(A_j)\|\,\colon\,
n\in\mathbb N \text{ and } A_1,\ldots,A_n\in\mathcal I \text{ are pairwise disjoint}
\right\}.
\]
When $E=\mathbb K$, we simply write $\ba(\mathcal I)\doteq\ba(\mathcal I,\mathbb K)$ and set:
\[
\ba_0(\mathcal I)
\doteq
\{m\in\ba(\mathcal I)\,\colon\,m(\{n\})=0\text{ for every }n\in\mathbb N\}.
\]
\end{definition}

The reader may notice that the space $\big(\ba(\mathcal I,E), \|\,\cdot\,\|\big)$ is a Banach space whenever $E$ is Banach.

\begin{theorem}\label{thm:dual-ba}
Let $\I$ be an ideal on $\omega$. The map
\begin{align*}
\mathfrak T\colon c_{0,\mathcal I}^*&\longrightarrow\ba(\mathcal I),\\
\Phi&\longmapsto m_\Phi
\end{align*}
where $m_\Phi(A)\doteq\Phi(\chi_A)$ for all $A\in\mathcal I$ is a surjective isometry.
\end{theorem}
\begin{proof}
Let $\Phi\in c_{0,\mathcal I}^*$ be fixed and consider the map $m_\Phi\colon\mathcal I\to\mathbb K$ defined above. Clearly, $m_{\Phi}(A\cup B)=m_{\Phi}(A)+m_{\Phi}(B)$ whenever $A,B\in\mathcal I$ are disjoint.

On the other hand, let $n\in\mathbb N$. If $A_1,\ldots,A_n\in\mathcal I$ are given and pairwise disjoint, by writing $|m_\Phi(A_j)|=a_j m_\Phi(A_j)$ with $|a_j|=1$ for $j=1,\ldots,n$, we have:
\begin{gather*}
    \sum_{j=1}^n|m_{\Phi}(A_j)|=\sum_{j=1}^na_jm_{\Phi}(A_j)=\sum_{j=1}^na_j\Phi(\chi_{A_j})=\Phi\left(\sum_{j=1}^na_j\chi_{A_j}\right)\leq\|\Phi\|.
\end{gather*}
Hence $\|m_{\Phi}\|\leq\|\Phi\|$. Now we will prove the reverse inequality. Let $\varepsilon>0$ be given. Then there is ${\bf x}\in B_{c_{0,\mathcal I}}$ such that
$\|\Phi\|-\varepsilon<|\Phi({\bf x})|$. By Remark \ref{denseness}, there exist $A_1,\ldots,A_n\in\mathcal I$ and $a_1,\ldots,a_n\in\mathbb K$ such that $\|{\bf x-y}\|<\varepsilon$, where ${\bf y}=\sum_{j=1}^na_j\chi_{A_j}$. By a standard argument, there are $B_1,\ldots,B_m\in\mathcal I$ disjoint and $b_1,\ldots,b_m\in\mathbb K$
with $|b_i|\leq1+\varepsilon$ for all $i=1,\ldots,m$ satisfying ${\bf y}=\sum_{i=1}^mb_i\chi_{B_i}$. Therefore, 
\begin{align*}
    \|\Phi\|-\varepsilon<|\Phi({\bf x})|
    &\leq\varepsilon\|\Phi\|+|\Phi({\bf y})|\\
    &\leq\varepsilon\|\Phi\|+\left|\Phi\left(\sum_{i=1}^mb_i\chi_{B_i}\right)\right|\\
    &\leq\varepsilon\|\Phi\|+\sum_{i=1}^m|b_i\Phi(\chi_{B_i})|\\
    &\leq\varepsilon\|\Phi\|+(1+\varepsilon)\sum_{i=1}^m|\Phi(\chi_{B_i})|\\
    &\leq\varepsilon\|\Phi\|+(1+\varepsilon)\|m_{\Phi}\|.
\end{align*}
Since $\varepsilon>0$ was arbitrary, we conclude that $\|\Phi\|\leq\|m_{\Phi}\|$. Thus, $\mathfrak T$ is an isometry.

    It remains to prove that $\mathfrak T$ is onto. Let $m\in ba(\mathcal I)$ be given. Define $\Phi\colon c_{00,\mathcal I}\to\mathbb K$ by
    $\Phi(\sum_{i=1}^kc_i\chi_{A_i})=\sum_{i=1}^kc_im(A_i)$ if $A_1,\ldots,A_k$ are pairwise disjoint. If ${\bf x}=\sum_{j=1}^na_j\chi_{A_j}\in c_{00,\mathcal I}$ has norm less than 1, then 
    \begin{gather*}
        |\Phi({\bf x})|=\left|\sum_{j=1}^na_jm(A_j)\right|\leq \sum_{j=1}^n|a_j||m(A_j)|\leq\|m\|.
    \end{gather*}
    Hence, $\Phi$ is continuous. By following the argument used at the beginning of the proof, we obtain $\|\Phi\|=\|m\|$. So, $\Phi$ admits a unique norm-preserving extension to $c_{0,\mathcal I}$. Clearly, $\mathfrak T(\Phi)=m_{\Phi}=m$.
\end{proof}

Given $(x_n)_{n\in \N} \in\ell_\infty(X)$ and $A\subset \N$, we define a new element $\chi_A\cdot(x_n)_n \in \ell_\infty(X)$, whose coordinates are given by:
\[
\pi_i\Big(\chi_A\cdot(x_n)_n\Big) =  \begin{cases}
    x_i,& \text{if } i\in A\\ 
    0,& \text{if }i\notin A.
\end{cases}
\]
\begin{theorem}\label{thm: dual c0I(X)}
    Let $X$ be a Banach space and $\mathcal I$ an ideal. The operator 
    \begin{align*}
        S: c_{0,\mathcal I}(X)^*&\longrightarrow \ba(\mathcal I, \ell_\infty(X)^*),\\
        \varphi&\longmapsto m_\varphi
    \end{align*}
    where $m_\varphi(A)\Big((x_n)_n\Big) \doteq \varphi \Big(\chi_A\cdot (x_n)\Big)$ for all $A\in \mathcal I,(x_n) \in \ell_\infty(X)$, is a (non-surjective) isometry.
\end{theorem}
\begin{proof}
    Let $\varphi \in c_{0,\mathcal I}(X)^*$. We will argue that $\|m_\varphi\| \leq \|\varphi\|$:

    Let $A_1,\ldots,A_n\in\mathcal I$ be pairwaise disjoint and $z_j \in B_{\ell_\infty(X)}$, $j=1,\ldots,n$. Then we have:
    \begin{align*}
        \sum_{j=1}^n |m_\varphi(A_j)(z_j)| &=
            \sum_{j=1}^n  \alpha_j \varphi(\chi_{A_j}\cdot z_j)\\
            &= \varphi\Big( \sum_{j=1}^n \underbrace{\alpha_j \chi_{A_j}\cdot z_j}_{\in B_{c_{0,\mathcal I}(X)}}\Big)\\
            &\leq \|\varphi\|   
    \end{align*} 

Taking the supremum over $z_j \in B_{\ell_\infty(X)}$ for $1\leq j\leq n$ yields
\[
\sum_{j=1}^n \|m_\varphi(A_j)\| \leq \|\varphi\|,
\]
and taking the supremum over finite unions of disjoint sets in $\mathcal I$ yields $\|m_\varphi\|\leq \|\varphi\|$.

On the other hand, to show that $\|\varphi\|\leq\|m_\varphi\|$, let $x\in B_{c_{0,\mathcal I}(X)}$ and fix $\varepsilon>0$. Since $c_{0,\mathcal I}(X) = \overline{\spn}^{\|\cdot\|_\infty}\{\chi_A\cdot{\bf y}: A\in \mathcal I, {\bf y} \in \ell_\infty(X) \}$ (see \cite[Proposition 2.5]{rincon-uzcategui}), choose $z=\sum\limits_{j=1}^m \chi_{A_i}\cdot y_i$ such that $\| x-z\|<\varepsilon$ and observe that:
\begin{align*}
\left| \varphi \left( \sum_{j=1}^m \chi_{A_j}\cdot y_j\right)\right| &=
    \left| \varphi \left( \sum_{j=1}^m \chi_{B_j}\cdot \widehat{y_j}\frac{(1+\varepsilon)}{(1+\varepsilon)}\right )\right|\\
    &= (1+\varepsilon) \left| \varphi \left( \sum_{j=1}^{l_m} \chi_{B_j}\cdot \frac{\widehat{y_j}}{(1+\varepsilon)}\right)\right|\\
    &\leq (1+\varepsilon)\sum_{j=1}^{l_m}\Big|\varphi \Big(\underbrace{\chi_{B_j}\cdot \frac{\widehat{y_j}}{(1+\varepsilon)}}_{m_\varphi(B_j)\Big( \frac{\widehat{y_j}}{(1+\varepsilon)}\Big)}\Big)\Big| \\
    &\leq (1+\varepsilon) \sum_{j=1}^{l_m}\| m_\varphi(B_j)\|\\
    &\leq (1+\varepsilon)\|m_\varphi\|,
\end{align*}
where the sets $B_j\in \mathcal I$ are disjoint and each $\|\widehat{y_j}\|\leq 1+\varepsilon$.

Notice that we have:
\[
|\varphi(x)| \leq |\varphi(x-z)| + |\varphi(z)| \leq \|\varphi\|\varepsilon + (1+\varepsilon)\|m_\varphi\|.
\]

Taking the limit with $\varepsilon\to 0$ and the supremum over $B_{c_{0,\mathcal I}(X)}$ yields $\|\varphi\| \leq \|m_\varphi\|$.
\end{proof}

\begin{remark}
The operator $S$ is not onto in general. Indeed, the reader may notice that the range of $S$ is exactly:
\begin{equation}\label{eq:locality-vector-measures}
\ran(S) = \{m\in  \ba(\mathcal I, \ell_\infty(X)^*): m(A)({\bf x})=m(A)(\chi_A\cdot{\bf x}),
\, A\in\mathcal I,\, {\bf x}\in\ell_\infty(X) \}
\end{equation}


For $X=\mathbb K$, fix distinct $p,q\in\mathbb N$, let $\delta_q((x_n))=x_q$, and define 
$m(A)=
\begin{cases}
\delta_q,&p\in A,\\
0,&p\notin A
\end{cases}
$.

Then $m\in\ba(\mathcal I,\ell_\infty^*)$ and $\|m\|=1$, but
\[
m(\{p\})(\chi_{\{q\}})=1
\qquad\text{and}\qquad
m(\{p\})(\chi_{\{p\}}\cdot\chi_{\{q\}})=0.
\]
Thus $m$ does not belong to the range of $S$.
\end{remark}

Having established these measure-theoretic descriptions, we now turn to two
classical sequence-space approaches to duality. We first consider
$\ell_{1,\mathcal I}$, defined by the $\mathcal I$-convergence of the partial
sums of $\sum_n |x_n|$, and then determine the Köthe--Toeplitz duals of
$c_{0,\mathcal I}$.

Given an ideal $\mathcal I$, consider:
\[\ell_{1,\mathcal I}
\doteq
\left\{
{\bf x}=(x_n)\in\mathbb K^{\mathbb N}\,\colon\, 
\sum_{j=1}^\infty|x_j| \text{ is $\mathcal I$-convergent}
\right\},
\]
endowed with the norm:
\[
\|{\bf x}\|_{\mathcal I}
\doteq
\mathcal I\text{-}\lim_n\sum_{j=1}^n|x_j|.
\]

The reader may notice that the operator $J_{\mathcal I}\colon\ell_{1,\mathcal I}\longrightarrow c_{0,\mathcal I}^*$ defined by:
\[
J_{\mathcal I}({\bf y})({\bf x}) \doteq 
\sum_{n=1}^\infty x_ny_n =
\mathcal I\text{-}\lim_N\sum_{n=1}^N x_ny_n,
\qquad \forall {\bf y}=(y_n)\in\ell_{1,\mathcal I},\,
{\bf x}=(x_n)\in c_{0,\mathcal I}
\]
is well defined (see \cite[Lemma 3.2]{balcerzak}) and is a linear isometric embedding. On the other hand, the next result shows that this approach yields no new sequence space.

\begin{proposition}
If $\mathcal I$ is a proper ideal, then $\ell_{1,\mathcal I} = \ell_1$.
\end{proposition}

\begin{proof}
The inclusion $\ell_1\subseteq\ell_{1,\mathcal I}$ is immediate, since ordinary convergence implies $\mathcal I$-convergence. Conversely, let ${\bf x}\in\ell_{1,\mathcal I}$ and denote:
\[
S_n=\sum_{j=1}^n|x_j|
\qquad\text{and}\qquad
\alpha=\mathcal I\text{-}\lim_n S_n.
\]
If $(S_n)$ were unbounded, then, since it is increasing, there would be $n_0\in\mathbb N$ such that $S_n>\alpha+1$ for every $n\geq n_0$. Hence
\[
[n_0,+\infty)\cap\N
\subseteq
\{n\in\mathbb N\,\colon\,|S_n-\alpha|\geq1\} \in\mathcal I,
\]
which contradicts our assumption that $\I$ is proper. Thus $(S_n)$ is
bounded and increasing, so it converges in the ordinary sense. By uniqueness of ideal limits, it follows that:
\[
\|{\bf x}\|_1=\lim_nS_n=\alpha=\|{\bf x}\|_{\mathcal I}.
\]

Hence, $\ell_{1,\I} = \ell_1$.
\end{proof}

Although $\ell_{1,\mathcal I}$ introduces no new sequence space, its definition recovers the canonical copy of $\ell_1$ in $c_{0,\mathcal I}^*$. Theorem~\ref{thm:dual-splitting} will show that $J_{\mathcal I}(\ell_1)$ is, in fact, the range of a projection of norm one on $c_{0,\mathcal I}^*$.

We next turn to another classical notion of duality for sequence spaces, namely the Köthe--Toeplitz duals. For a sequence space $E$, its associated $\alpha$-, $\beta$-, and $\gamma$-duals are defined by:
\begin{align*}
E^{\alpha} &=
\left\{ x\in\mathbb K^{\N}: 
\sum_{i=1}^{\infty}|x_i y_i|<\infty \text{ for every }y\in E
\right\},\\
E^{\beta} &=
\left\{ x\in\mathbb K^{\N}:
\sum_{i=1}^{\infty}x_i y_i \text{ converges for every }y\in E
\right\},\\
E^{\gamma} &=
\left\{ x\in\mathbb K^{\N}:
\sup_{n\in\N} \left|\sum_{i=1}^{n}x_i y_i\right|<\infty \text{ for every }y\in E
\right\}.
\end{align*}

It follows directly from the definitions that $c_{00}\subseteq E^\alpha\subseteq E^\beta\subseteq E^\gamma$. Moreover, since $c_{0,\mathcal I}\subseteq\ell_\infty$, we have $\ell_1\subseteq(c_{0,\mathcal I})^\alpha$. The Köthe--Toeplitz duals are particular cases of the multiplier sequence spaces $M(E,F) = \left\{ u\in\mathbb K^\N: (u_i x_i)_i\in F\text{ for every }x\in E \right\}$, where $E,F,$ are sequence spaces (see \cite[4.3, Theorem 15]{WilanskySummability}). For a historical account of these notions, we refer to \cite{GarlingTopologicalsequence}; a more detailed introduction can be found in \cite[Section 30.1]{kothetopvec1}.

In \cite{Garlingbetagamadual}, the author determines the Köthe--Toeplitz duals of several sequence spaces introduced in \cite{GarlingTopologicalsequence}. Interestingly, all these spaces have the same $\alpha$-dual, namely $\ell_1$. This suggests investigating whether the same phenomenon occurs for the spaces $c_{0,\mathcal I}$.

We first record a simple monotonicity observation: if $\mathcal I$ and $\mathcal J$ are ideals such that $\mathcal I\subseteq\mathcal J$, then $c_{0,\mathcal I}\subseteq c_{0,\mathcal J}$ and $(c_{0,\mathcal J})^\delta\subseteq(c_{0,\mathcal I})^\delta$ for every $\delta\in\{\alpha,\beta,\gamma\}$. In particular, $(c_{0,\mathcal I})^\alpha\subseteq(c_0)^\alpha$. The following elementary fact will show that
$(c_{0,\mathcal I})^\alpha=\ell_1$. We include a short proof for completeness.

\begin{proposition} 
$(c_0)^{\alpha} = \ell_1$.
\end{proposition}
\begin{proof}
    Suppose towards a contradiction that there exists $(x_n)\in (c_0)^{\alpha}\setminus \ell_1$.

    Given $k \in \N$, there exists $m_k\in \N$ such that $\sum_{i=1}^{m_k} |x_i| > k$. Take a partition of $\N=\bigcup_{n} I_n$ in consecutive intervals such that $\sum_{i \in I_n}|x_i|>n$ and consider the sequence $y = (y_i)$  defined by $y_i = \tfrac{1}{n}$ if $i\in I_n$. It is clear that $y\in c_0$. Moreover, it is easy to check that $\sum_{n=1}^{m_k}|x_n y_n| >k$ for every $k\geq 1$, a contradiction.
\end{proof}

It is easy to check that if $E$ is a solid space (downward closed with respect to a partial order $\leq$ on $E$), then all Köthe-duals are equal. In particular, for $c_{0,\mathcal I}$, we obtain:

\begin{corollary}
    If $\mathcal I$ is an ideal, then $(c_{0,\mathcal I})^\delta = \ell_1$ for all $\delta \in \{\alpha, \beta,\gamma\}$.
\end{corollary}

Finally, we return to the measure representation. We show that every $m\in\ba(\mathcal I)$ admits a canonical decomposition as the sum of a measure determined by its values on the singletons and a measure in $\ba_0(\mathcal I)$. As a consequence, the canonical copy of $\ell_1$ in $c_{0,\mathcal I}^*$ is the range of a projection of norm one. This yields the isometric decomposition stated in Theorem~\ref{thm:dual-splitting}.

Recall that $c_{0,\mathcal I} \equiv C_0(U_\mathcal I)$ as Banach lattices (Proposition \ref{isometrias}). Thus, the Riesz Representation Theorem yields $c_{0,\I}^*\equiv M(U_\I)$.

 For a finite Radon measure $\mu$ on $U_\I$, define:
\[
 \forall B\in\bor(U_\mathcal I):\quad
 \mu_{\omega}(B)\doteq\mu(B\cap\omega) \quad \text{and} \qquad
 \mu_F(B)\doteq\mu\big(B\cap(U_{\mathcal I}\setminus\omega)\big)
 .
\]
These restrictions are also Radon measures and are concentrated on the disjoint Borel sets $\omega$ and $F_\I= U_\I\setminus\omega$, respectively, thus $\mu_{\omega}$ and $\mu_F$ are mutually singular.  Their variation measures are concentrated on the same sets, thus: 
\begin{equation}\label{eq:measure-l1-split}
 \mu=\mu_{\omega}+\mu_F,
 \qquad
 |\mu|=|\mu_{\omega}|+|\mu_F|,
 \qquad
 \|\mu\|=\|\mu_{\omega}\|+\|\mu_F\|.
\end{equation}

\begin{theorem}\label{thm:dual-splitting}
For every ideal $\I$:
\begin{equation}\label{eq:dual-splitting}
 c_{0,\I}^* 
 \equiv \ba(\mathcal I)
 \equiv \ell_1\oplus_1 M(U_\I\setminus\omega)
 \equiv M(\omega) \oplus_1 M(U_\I\setminus\omega)
 \equiv \ell_1\oplus_1\ba_0(\I)
\end{equation}
isometrically.
\end{theorem}

\begin{proof}
Let $m\in\ba(\I)$ and put $a_n=m(\{n\})$ for each $n\in\omega$. If $F\subseteq \omega$ is finite, then:
\[
 \sum_{n\in F}|a_n|=\sum_{n\in F}|m(\{n\})|\leq\|m\|.
\]
Taking the supremum over $F$ yields $a=(a_n)\in\ell_1$ with $\|a\|_1\leq\|m\|$.

\begin{claim}\label{claim: definition of m_a}
    $m_a:\mathcal I\to\R$ defined by $m_a(A)\doteq \sum_{n\in A} a_n$ is such that $m_a\in \ba(\mathcal I)$ and $\|m_a\| = \|a\|_1$. \bigskip

    Finite additivity is clear and $m_a$ is well defined because $\sum_n |a_n|$ converges absolutely. Thus, for $A_1,\ldots,A_r\in\I$ pairwise disjoint, it follows that:
\[
 \sum_{j=1}^r|m_a(A_j)|
 \leq\sum_{j=1}^r\sum_{n\in A_j}|a_n|
 \leq\|a\|_1.
\]
Hence, $m_a\in\ba(\I)$ and $\|m_a\|\leq\|a\|_1$. Conversely, one can prove that $\|a\|_1\leq \|m_a\|$ by applying the variation norm of $m_a$ on the sets $\{0,\ldots,n\}$.
\end{claim}

Observe that $m_0\doteq m-m_a \in \ba_0(\mathcal I)$. Indeed, for each $n\in\omega:$
\[
m_0(\{n\})\doteq m(\{n\})-m_a(\{n\}) = a_n -a_n=0.
\]

It remains to prove the norm identity.  Let $\mu\in M(U_\I)$ correspond to $m$ under Theorem~\ref{thm:dual-ba}. For $A\in\mathcal I$, the element  $\chi_A\in c_{0,\mathcal I}$ corresponds to $\widehat A\subseteq U_\I$ in $C_0(U_\mathcal I)$ in \cite[Proposition 3.1]{rincon-uzcategui}, thus, $ m(A)=\mu(\widehat A)$ for every $A\in\I$. In particular, $\mu(\{n\})=m(\{n\})=a_n$.  Therefore:
\[
 \mu_{\omega}=\sum_{n\in\omega} a_n\delta_n
 \qquad \text{and}\qquad
 \|\mu_{\omega}\|=\sum_{n\in\omega}|a_n|=\|a\|_1.
\]

Because a principal ultrafilter $n$ belongs to $\widehat A$ exactly when $n\in A$, we have: 
\[\forall A\in\mathcal I:\quad\mu_{\omega}(\widehat A)=\sum_{n\in A}a_n=m_a(A).\]

On the other hand, for each $A\in\I$, $\mu_F(\widehat A) = \mu(\widehat A)- \mu_\omega(\widehat A) = m(A) - m_a(A)$. Therefore, $m_0$ corresponds to $\mu_F$ and, by Equation~\eqref{eq:measure-l1-split}:
\[ \|m\|=\|\mu\|  =\|\mu_{\omega}\|+\|\mu_F\|  =\|a\|_1+\|m_0\|\]

It is straightforward to check that the map $\Lambda:\ba(\mathcal I)\to \ell_1\oplus_1\ba_0(\mathcal I)$ given by $\Lambda(m)=(a, m_0)$, where $a=\big(m(\{n\})\big)_{n\in\omega}$ and $m_0=m-m_a$ is surjective. The norm identity above shows that $\Lambda$ is an isometry. Therefore, $\ba(\I)\equiv\ell_1\oplus_1\ba_0(\I)$.

On the other hand, the mutually singular decomposition $\mu=\mu_\omega+\mu_F$ in Equation~\eqref{eq:measure-l1-split} and the canonical isometric identification $M(\omega)\equiv\ell_1$ give $M(U_\I)\equiv M(\omega)\oplus_1M(U_\I\setminus \omega)\equiv \ell_1\oplus_1M(U_\I\setminus \omega)$. Combining these identifications with $ c_{0,\I}^*\equiv\ba(\I)\equiv M(U_\I)$ proves both isometric identifications in
\eqref{eq:dual-splitting}.
\end{proof}

The isometric identifications established in the previous proof can be summarized in the following commutative diagram:
\[
\begin{tikzcd}[column sep=6em, row sep=4em]
c_{0,\I}^*\equiv \ba(\mathcal I)
\arrow[r, "\Phi"]
\arrow[d, "\Lambda"']
&
M(U_{\mathcal I})
\arrow[d, "\Gamma"]
\\
\ell_1\oplus_1\ba_0(\mathcal I)
&
M(\omega)\oplus_1M(U_{\mathcal I}\setminus\omega)
\arrow[l, "\Psi"']
\end{tikzcd}
\]

where $\Phi$ is the isometric identification given by the Riesz Theorem and $c_{0,\mathcal I}\equiv C_0(U_\mathcal I)$ that associates:
\[
\Phi(m) = \mu \qquad\text{and}\qquad m(A)=\mu(\widehat A), \,\forall A\in\mathcal I.
\]

The map $\Gamma:M(U_\mathcal I)\to M(\omega)\oplus_1 M(U_\mathcal I\setminus\omega)$, given by $\Gamma(\mu) = (\mu_\omega,\mu_F)$, records the two mutually singular components of $\mu$ isometrically via Equation~\eqref{eq:measure-l1-split} and $\Lambda(m) = (a,m_0)$, where:
\[
a = \big(m(\{n\})\big)_{n\in\omega}\qquad \text{and}\qquad m_0 = m-m_a
\]
with $m_a$ defined in Claim~\ref{claim: definition of m_a}. 

Finally, the map $\Psi: M(\omega)\oplus_1 M(U_\mathcal I\setminus\omega)\to \ell_1 \oplus_1 \ba_0(\mathcal I)$ is the isometric identification given by $\Psi(\mu_\omega,\mu_F) = (a,m_0)$, whose inverse is $\Psi^{-1}(a, m_0) = (\mu_\omega, \mu_{m_0})$ where $\mu_\omega \doteq \sum_{n\in\omega} a_n\delta_n$ and $\mu_{m_0}$ is the measure corresponding to $m_0$ under the Riesz identification $\Phi$.
\section{Compact subsets of $c_{0,\mathcal I}$}
\label{section compactness}

In this section, we establish an ideal version of Dini's theorem in Theorem~\ref{thm:ideal-dini}. We then use this result to characterize the relatively compact subsets of $c_{0,\mathcal I}$
in Theorem~\ref{thm:compactness-c0I}, extending the classical
criterion for $c_0$.

For $A\in\mathcal I$, we consider $q_A: c_{0,\mathcal I}\rightarrow \R$, defined by: 
\[q_A(x)=\sup_{n\notin A}|x_n|.\]

\begin{theorem}[Uniform ideal tails]\label{thm:ideal-dini}
If $L$ is a compact subset of $c_{0,\mathcal I}$, then:
\[\inf_{A\in\mathcal I}\sup_{x\in L}q_A(x)=0.\]
Equivalently, the decreasing net $ \bigl(q_A|_L\bigr)_{A\in(\mathcal I,\subseteq)}$ converges uniformly to zero.
\end{theorem}

\begin{proof}
Observe that each $q_A$ is $1$-Lipschitz, and that $q_{A\cup B}\leq\min\{q_A,q_B\}$ for every $A,B\in\I$.

Fix $\varepsilon>0$.  For every $x\in c_{0,\mathcal I}$, $A(x,\varepsilon/2)=\left\{n\in\omega:|x_n|\geq\frac{\varepsilon}{2}\right\} \in \I$, and $q_{A(x,\varepsilon/2)}(x)\leq\frac{\varepsilon}{2}<\varepsilon$. Because $q_A$ is continuous, $U_A\doteq\{y\in L:q_A(y)<\varepsilon\}$ is open in $L$ and $(U_A)_{A\in\I}$ covers $L$. By compactness, $L \subseteq \bigcup_{i=1}^n U_{A_i}$. Then $A=A_1\cup\cdots\cup A_n \in \I$ and for each $x\in L$, some $j\leq n$ satisfies $q_{A_j}(x)<\varepsilon$, hence $q_A(x)\leq q_{A_j}(x)<\varepsilon$.
\end{proof}
\begin{remark}
When $\mathcal I=\fin$, the sets $A_N=\{1,\ldots,N\}$ form a sequence in $(\fin,\subseteq)$, and $q_{A_N}$ is just the tail of a sequence:
\[
q_{A_N}(x)=\sup_{n>N}|x_n|.
\]

For every $x\in c_0$, the sequence $\bigl(q_{A_N}(x)\bigr)_N$ decreases to zero. Hence, when $L$ is compact, the conclusion of Theorem~\ref{thm:ideal-dini} also follows directly from the Dini's classical theorem applied to the decreasing sequence of continuous functions $\bigl(q_{A_N}|_L\bigr)_N$. Thus, in the case $\mathcal I=\fin$, the theorem recovers the usual uniform-tail property of compact subsets of $c_0$ stated in Proposition~\ref{prop: rel compact c0}.
\end{remark}

Theorem~\ref{thm:ideal-dini} shows that compact families are uniformly $\mathcal I$-null: for every $\varepsilon>0$, their coordinates are smaller than $\varepsilon$ outside a common member of $\mathcal I$. This condition alone, however, does not imply relative compactness when $\mathcal I$ contains an infinite set:

\begin{example}
Let $A \in \mathcal I$ and define $K_A \doteq \{ x\in c_{0,\mathcal I}: \| x\| \leq \chi_A \}\subseteq c_{0,\mathcal I}$. Observe that $K_A$ is bounded and $K_A = \prod_{n\in \N} K_n$, where $K_n = \begin{cases}
    0,& \text{if } n\notin A\\
    [-1,1],& \text{if } n\in A
\end{cases}$ is isomorphic to $B_{\ell_\infty(A)}$. 

Thus, $K_A$ is not compact when $A$ is infinite.
\end{example}

The following result is known (see \cite[Lemma 4]{Sargent1966}). We include a proof here for the sake of completeness.

\begin{lemma}\label{lem:finite-partition}
Let $S$ be a set and let $H$ be a bounded subset of $\ell_\infty(S)$.  The
following are equivalent.
\begin{enumerate}[(i)]
\item $H$ is relatively compact in $\ell_\infty(S)$.
\item For every $\varepsilon>0$, there is a finite partition $S=S_1\mathbin{\dot\cup}\cdots\mathbin{\dot\cup}S_m$ such that:
\[
 \sup_{h\in H}\sup_{s,t\in S_j}|h(s)-h(t)|<\varepsilon,
 \qquad j=1,\ldots,m.
\]
\end{enumerate}
\end{lemma}

\begin{proof}
$(i)\Rightarrow(ii):$ Assume first that $H$ is relatively compact. Given $\varepsilon>0$, choose $0<\delta<\dfrac{\varepsilon}{4}$. By compactness of $\cl(H)$, there exist $h^1,\ldots,h^r \in H$ such that $H\subseteq \cl(H) \subseteq \bigcup_{i=1}^r B(h^i,\delta)$. Notice that for each $1\leq i\leq r$, $h^i[S]$ is a relatively compact set of $\K$ and by compactness, there exist $\alpha^i_1,\ldots,\alpha^i_{n_i}\in\K$ such that $h^i[S]\subseteq \bigcup_{j=1}^{n_i} B_\K(\alpha^i_j, \delta)$. Consider the sets:

\[
S_j^i \doteq \{s\in S: |h^i(s)-\alpha_j^i|<\delta \}, \quad \text{for } 1\leq i\leq r, 1\leq j\leq n_i.
\]

Without loss of generality, we can assume that for each $i\leq r$, $\mathcal P_i=\{S^i_j: 1\leq j\leq n_i\}$ is a finite partition of $S$. We take their common refinement:
\[
\mathcal P \doteq \left\{ \bigcap_{i=1}^r S^i_{j_i} : 1\leq j_k\leq n_k, \text{for } k\in\{1,\ldots,r\}\right\} \setminus \{\emptyset\}.
\]

It is clear that $\mathcal P$ is a partition of $S$. Moreover, observe that for each $s,s'\in \bigcap_{i=1}^r S^i_{j_i}$, we have:
\[
\forall i\leq r: \quad |h^i(s)-h^i(s')|\leq |h^i(s)-\alpha^i_{j_i}|+ |\alpha^i_{j_i} -h^i(s')| < 2\delta.
\]

Given $h\in H$, then $h\in B(h^i, \delta)$ for some $i\leq r$. Then, for all $s,s'\in \bigcap_{i=1}^r S^i_{j_i}$:
\begin{align*}
    |h(s)-h(s')| &\leq  |h(s) - h^i(s)| + |h^i(s) - h^i(s')| + |h^i(s')-h(s')| \\
        & \leq \delta + 2\delta +\delta<\varepsilon.
\end{align*}
 Thus, $\sup_{h\in H}\sup_{s,t\in S_\ell}|h(s)-h(t)|<\varepsilon$.

\medskip
$(ii)\Rightarrow(i):$ Choose one point $s_j\in S_j$ for each $j\in\{1,\ldots,m\}$ and define $R(h)\doteq\sum_{j=1}^m h(s_j)\chi_{S_j}$. Then $\|h-R(h)\|_\infty<\varepsilon$ for every $h\in H$.  The set $R(H)$ is bounded in the finite-dimensional space $\operatorname{span}\{\chi_{S_1},\ldots,\chi_{S_m}\}$, hence it is totally bounded.  Therefore $H$ is totally bounded.
\end{proof}

For a nonempty bounded set $K\subseteq c_{0,\mathcal I}$, define $a_K(n)=\sup_{x\in K}|x_n|$, for each $n\in\omega$. 

Given a set $A\in\mathcal I$ and a bounded set $K\subseteq c_{0,\mathcal I}$, we define $P_A(K)\doteq \{ x|_A \in \ell_\infty (A): x\in K\}$. With this notation, we are now able to present a characterization of the relatively compact subsets of $c_{0,\mathcal I}$:

\begin{theorem}
\label{thm:compactness-c0I}
Let $K\subseteq c_{0,\mathcal I}$ be nonempty.  The following conditions are
equivalent.
\begin{enumerate}[(1)]
\item $K$ is relatively compact in $c_{0,\mathcal I}$.

\item $K$ is bounded, $a_K\in c_{0,\mathcal I}$, and $P_A(K)$ is relatively compact in $\ell_\infty(A)$ for every $A\in\mathcal I$.

\item There is $b\in c_{0,\mathcal I}^{+}$ such that $|x|\leq b$ for each $x\in K$ and, for every $t>0$, $P_{B_t}(K)\subseteq\ell_\infty(B_t)$ is relatively compact, where $B_t\doteq\{n\in\omega:b_n\geq t\}$.

\item $K$ is bounded and, for every $\varepsilon>0$, there are $A\in\mathcal I$ and a finite partition $A=A_1\mathbin{\dot\cup}\cdots\mathbin{\dot\cup}A_m$ such that:
\[
\sup_{x\in K}\sup_{n\notin A}|x_n|<\varepsilon \qquad \text{ and }\qquad
\sup_{x\in K}\sup_{n,k\in A_j}|x_n-x_k|<\varepsilon
 \quad(j=1,\ldots,m).
\]
\end{enumerate}
\end{theorem}

\begin{proof}
(1)$\Rightarrow$(2): Notice that for $L=\overline K$, Theorem~\ref{thm:ideal-dini} implies that, for every $\varepsilon>0$, there is $A\in\mathcal I$ such that:
$\displaystyle \sup_{x\in K}\sup_{n\notin A}|x_n|\leq \sup_{x\in L}\sup_{n\notin A}|x_n|<\varepsilon.$ Thus, 
\[
\sup_{n\not\in A}a_K(n)=\sup_{n\not\in A}\sup_{x\in K}|x_n|=\sup_{x\in K}\sup_{n\notin A}|x_n|<\varepsilon.
\]
It follows that $\{n\in\omega:a_K(n)\geq\varepsilon\}\subseteq A\in\mathcal I$ and $a_K\in c_{0,\mathcal I}$.  Since every $P_A$ is continuous, $P_A(K)$ is relatively compact.  

\bigskip
(2)$\Rightarrow$(3): Take $b=a_K$. Then, $b\in c_{0,\mathcal I}^+$ and $|x_n|\leq b_n$ for each $x\in K$ and $n\in\omega$. The second part is immediate. 

\bigskip

$(3)\Rightarrow(1):$ Given $\varepsilon>0$, by our assumption $B\doteq B_{\frac{\varepsilon}{2}} = \{n:b_n\geq\varepsilon/2\}\in\mathcal I.$ By compactness of $\cl\big(P_B(K)\big)$, there exist $x^1,\ldots,x^j\in\ell_\infty(B)$ such that $P_B(K)\subseteq \bigcup_{i=1}^j B(x^i,\frac{\varepsilon}{2})$. We extend the centers of this open cover as follows:
\[
\widetilde{x^i}(n) \doteq \begin{cases}
    0, &\text{if } n\notin B\\
    x^i(n), &\text{otherwise}.
\end{cases}
\]

It is easy to check that $\widetilde{x^i}\in c_{0,\mathcal I}$.

\begin{claim}
    $K\subseteq \bigcup_{i=1}^j B\big(\widetilde{x^i},\frac{\varepsilon}{2}\big)$.

    Indeed, for each $x\in K$, observe that $x|_B\in P_B(K)$, thus $x|_B\in B\big(\widetilde{x^i},\frac{\varepsilon}{2}\big)$ for some $i$ and we have:
    \[
    |x(n) - \widetilde{x^i}(n)| = 
    \begin{cases}
    |x(n)|, &\text{if } n\notin B\\
    |x(n) - x^i(n)|, &\text{otherwise}
\end{cases}
    \]
    and in both cases we have $|x(n) - \widetilde{x^i}(n)|<\frac{\varepsilon}{2}$. Thus $\|x- \widetilde{x^i}\|\leq \frac{\varepsilon}{2}$.
\end{claim}

By the previous claim, $K$ is relatively compact in $c_{0,\mathcal I}$.

\bigskip

$(1)\Rightarrow (4)$: Use Theorem~\ref{thm:ideal-dini} to choose
$A\in\mathcal I$ with uniformly small coordinates outside $A$.  The set
$P_A(K)$ is relatively compact, so Lemma~\ref{lem:finite-partition}
supplies the required finite partition of $A$.

\bigskip
$(4)\Rightarrow(1)$: Follows directly from Lemma~\ref{lem:finite-partition}.
\end{proof}

Under the isometric identification
$c_{0,\mathcal I}\equiv C_0(U_{\mathcal I})$ established in Proposition \ref{isometrias},
Theorem~\ref{thm:compactness-c0I} can be viewed as a concrete form of the
Arzel\`a--Ascoli theorem for the locally compact space $U_{\mathcal I}$. The
uniform ideal-tail condition corresponds to equivanishing at infinity,
whereas the relative compactness of the restrictions $P_A(K)$ records the
behavior of the family on the compact clopen subsets
$\widehat A\subseteq U_{\mathcal I}$.

We conclude by recovering the classical compactness criterion for $c_0$
(see \cite[Exercise 1.50, p.~41]{fabian-et-al}):

\begin{proposition}\label{prop: rel compact c0}
    A subset $K$ of $c_0$ is relatively compact if and only if there is $(x_n)\in c_0$ such that  for each $k=(k_n)\in K$ we have $|k_n|\leq|x_n|$ for all $n\in\mathbb N$.
\end{proposition}

Observe that Proposition~\ref{prop: rel compact c0} is obtained from Theorem~\ref{thm:compactness-c0I} because when considering $\mathcal I=\fin$, the statement of the theorem is reduced to:
\[
 K\subseteq c_0\text{ is relatively compact}  \quad\Longleftrightarrow\quad  a_K\in c_0.
\]
Equivalently, $K$ is dominated coordinatewise by one element of $c_0$.

\section{On $c_{0,\mathcal I}$-valued operators}

\label{section sobzcyk}

\subsection{Bounded and compact $c_{0,\mathcal I}$-valued operators}

Let $\mathcal I$ be an ideal on $\mathbb N$. We denote the coordinate projections in $c_{0,\mathcal I}$ by $\pi_n:c_{0,\mathcal I}\rightarrow \mathbb K$.

\begin{remark}\label{the map between B and cI}
For a Banach space $X$, every bounded operator $T:X\rightarrow \ell_\infty$ is associated with the bounded sequence of functionals $(\alpha_n^*)_{n\in\mathbb N}$, where $\alpha_n^*\doteq \pi_n\circ T\in X^*$, with $\|T\|=\sup_{n\in\mathbb N}\|\alpha_n^*\|$. Moreover, observe that $\ran(T)\subseteq c_{0,\mathcal I}$ is equivalent to the condition 
 $\mathcal I-\lim_{n}\alpha_n^*(x) = 0$ for all $x\in X$. Furthermore, it is not difficult to prove that
    \[c_{0,\mathcal I}^{w^*}(X^*)\doteq \{(x_n^*)\in\ell_\infty(X^*): \mathcal I-\lim_n x_n^*(x)=0 \text{ for every } x\in X \}\] is a closed subspace of $(\ell_\infty(X^*),\|\cdot\|_\infty)$.
\end{remark}

\begin{theorem}
   Let $X$ be a Banach space and $\mathcal I$ be an ideal on $\mathbb N$. The map
 \begin{equation*}
       \Psi\colon B(X,c_{0,\mathcal I})\to c_{0,\mathcal I}^{w^*}(X^*),\qquad
        T\mapsto (\alpha_n^*)
 \end{equation*}
    is an onto linear isometry. Moreover, if $X$ is infinite-dimensional and $\mathcal I\neq{\sf Fin}$, then $\Psi(K(X,c_{0,\mathcal I}))\subsetneq c_{0,\mathcal I}(X^*)$.
\end{theorem}

\begin{proof}
  By Remark \ref{the map between B and cI}, the map $\Psi$ is well defined and continuous.  It remains to check that $\Psi$ is onto: If $(x_n^*)\in c_{0,\mathcal I}^{w^*}(X^*)$, we define $T\colon X\to c_{0,\mathcal I}$
    by $Tx=(x_n^*(x))$ for $x\in X$. Notice that $T$ is well defined, linear, continuous and $\Psi(T)=(x_n^*)$. 

   We will show that if  $T\in K(X,c_{0,\mathcal I})$, then $\Psi(T)=(\alpha_n^*)\in c_{0,\mathcal I}(X^*)$.

        Since $T\in K(X,c_{0,\mathcal I})$, $T(B_X)$ is a relatively compact subset of $c_{0,\mathcal I}$. Let $\varepsilon>0$ be given. Then, there are $y^1,\ldots,y^n\in c_{0,\mathcal I}$ such that $T(B_X)\subseteq\bigcup_{1\leq j\leq n}B(y^j,\varepsilon/3)$. For each $j\in\{1,\ldots,n\}$, $A_j=\{n\in\mathbb N\,\colon\,|y_n^j|\geq\frac{2\varepsilon}{3}\}\in\mathcal I$. Thus, $A=A_1\cup\cdots\cup A_n\in\mathcal I$. Now, fix $Tx=(y_n)\in T(B_X)$. Hence, there is $j\in\{1,\ldots,n\}$ such that $\|y-y^j\|<\varepsilon.$ If $n\not\in A$, then $ |y_n|\leq|y_n-y_n^j|+|y_n^j|<\dfrac{\varepsilon}{3}+\dfrac{2\varepsilon}{3}=\varepsilon.$  Consequently,
        \begin{equation*}
            \|\alpha_n^*\|=\sup_{x\in B_X}|\alpha_n^*(x)|= \sup_{x\in B_X}|\pi_n(Tx)|=\sup_{y\in T(B_X)}|y_n|<\varepsilon.
        \end{equation*}
        It follows that $\{n\in\mathbb N\,\colon\,\|\alpha_n^*\|\geq\varepsilon\}\subseteq A\in\mathcal I$.   

        Suppose that $X$ is infinite-dimensional and let $A\in\mathcal I\setminus\fin$. Let $(x_n)$ be a normalized basic sequence in $X$ and let $(x_n^*)$ be the sequence of Hahn-Banach extensions of the coordinate functionals associated to $(x_n)$.

        Write $A=\{n_1<n_2<\cdots n_k<\cdots\}$. Define
        \begin{equation*}
            y_n^*=\begin{cases}
                x_k^*, & n\in A\text{ and $n=n_k$ for some $k\in\mathbb N$;}\\
                0, & n\not\in A.
            \end{cases}
        \end{equation*}
        Notice that for each $\varepsilon>0$, $\{n\in\mathbb N\,\colon\,\|y_n^*\|\geq\varepsilon\}\subseteq A\in\mathcal I$. Let $T\colon X\to c_{0,\mathcal I}$ be given by
        $T(x)=(y_n^*(x))$, for all $x\in X$. We claim that $T\not\in K(X,c_{0,\mathcal I})$. Indeed, $\pi_{n_k}(Tx_m)=x_k^*(x_m)=\delta_{km}$ for each $k,m\in\mathbb N$.
        It follows that $\|Tx_k-Tx_m\|\geq 1$ for each $k\neq m$. Thus, $T\not\in K(X,c_{0,\mathcal I})$.
\end{proof}

\begin{theorem}\label{thrm:compact-operators}
Let $\I$ be an ideal on $\mathbb N$, $X$ be a Banach space, and $(\alpha_n^*)\in c_{0,\I}^{w^*}(X^*)$. The operator
\begin{align*}
        T:X &\longrightarrow c_{0,\mathcal I}\\
        x &\longmapsto (\alpha_n^*(x))_{n\in\omega}
    \end{align*}
is compact if and only if the following conditions hold:
\begin{enumerate}[(i)]
\item $(\|\alpha_n^*\|)_{n\in\omega}\in c_{0,\mathcal I}$; and
\item $\{\alpha_n^*:n\in\omega\}$ is relatively norm compact in $X^*$.
\end{enumerate}
\end{theorem}

\begin{proof}
If $T$ is compact, apply Theorem~\ref{thm:ideal-dini} to
$\overline{T(B_X)}$.  For every $\varepsilon>0$, there is 
$A\in\mathcal I$ such that
\[
 \sup_{n\notin A}\|\alpha_n^*\|
 =\sup_{x\in B_X}\sup_{n\notin A}|\alpha_n^*(x)|
 <\varepsilon.
\]
Thus, $\{n\in\omega:\|\alpha_n^*\|\geq\varepsilon\}\subseteq A$, and (i) holds.  Observe that
$\alpha_n^*=T^*\pi_n$ for all $n\in\omega$. By Schauder's theorem, $T^*$ is compact, and
(ii) follows.

Conversely, fix $\varepsilon>0$ and put $A=\{n\in\omega\,\colon\,\|\alpha_n^*\|\geq\varepsilon\}\in\mathcal I.$ By (ii), choose $x^*_1,\ldots,x^*_m\in X^*$ such that $\{\alpha_n^*\,\colon\,n\in A\}\subseteq\bigcup_{1\leq j\leq m}B(x_j^*,\varepsilon)$. Let $A_j=\{n\in A\,\colon\,\|\alpha_n^*-x_j^*\|<\varepsilon\}$, where $j=1,\ldots,m$.
Clearly, $A=\bigcup_{1\leq j\leq m}A_j$. Without loss of generality, we may assume that the sets $A_i$ are pairwise disjoint. Let
\[
S:X\to c_{0,\mathcal I},\quad Sx=\sum_{j=1}^mx^*_j(x)\chi_{A_j}, \quad x\in X.
\]
The operator $S$ has finite rank and
$\|T-S\|\leq\varepsilon$.  Hence $T$ is compact.
\end{proof}

\subsection{Sobczyk's theorem for \texorpdfstring{$c_{0,\mathcal I}$}{c0I} spaces}

The pursuit of generalizations of Sobczyk’s Theorem has been an active area of research over the past decades (\cite{argyros, correa2014compact, correa2013extensions,correa2015c0, molto, patterson, sanchez2006yet,atob2000sobczyk}). In these papers, several related notions have emerged, for instance:  when a pair of Banach spaces $Y\subseteq X$ satisfies the conclusion of Sobczyk's Theorem, we say that $Y$ has the $c_0$-extension property ($c_0$-EP) in $X$ (see \cite{correa2013extensions,correa2015c0}), in \cite{VT} the $c_0(I)$-EP is defined when $c_0$ is replaced by its non-separable version in the Sobczyk's theorem. From an homological standpoint, the classical Sobczyk's theorem says that $c_0$ is a Banach space that is separably injective.

\medskip

According to \cite[Definitions 1.1 and 2.1]{aviles-et-all}, a Banach space $E$ is
\begin{itemize}
    \item \textit{injective} if for every Banach space $X$ and every subspace $Y$ of $X$, each operator $T\in B(Y,E)$ admits an extension $\widetilde T\in B(X,E)$. 

    \item \textit{separably injective} if for every separable Banach space $X$ and every subspace $Y$ of $X$, each operator $T\in B(Y,E)$ admits an extension $\widetilde T\in B(X,E)$. 

    \item \textit{universally separably injective} if for every Banach space $X$ and every separable subspace $Y$ of $X$, each operator $T\in B(Y,E)$ admits an extension $\widetilde T\in B(X,E)$.
\end{itemize}

In each case, for $\lambda>0$, the corresponding $\lambda$-notion requires that the extension can always be chosen with $\|\widetilde T\|\leq\lambda\|T\|$.

 We will show that $c_{0,\mathcal I}$ is 2-separably injective and $\ell_\infty/c_{0,\mathcal I}$ is $1$-universally separably injective.
 
\begin{theorem}\label{Prop: c0I is separably injective}
    Let $\mathcal I$ be an ideal on $\mathbb N$. Then,  
    \begin{enumerate}[(1)]
        \item $c_{0,\mathcal I}$ is 2-separably injective.
        \item $\ell_\infty/c_{0,\mathcal I}$ is $1$-universally separably injective.
    \end{enumerate}
\end{theorem} 

\begin{proof}
(1) Since $\ell_\infty\cong C(\beta\mathbb N)$ is 1-injective, from \cite[Theorem 2.18]{aviles-et-all}
    we obtain that $C_0(U_\mathcal I)$ is 2-separably injective. By Proposition \ref{isometrias}
    we conclude that $c_{0,\mathcal I}$ is 2-separably injective.

(2) By \cite[Theorem 2.18]{aviles-et-all}, $C(K_\mathcal I)$ is 1-universally separably injective. It follows from Proposition \ref{isometrias}
    that $\ell_\infty/c_{0,\mathcal I}$ is 1-universally separably injective.
\end{proof}

Although we already know that the spaces $c_{0,\mathcal I}$ are 2-separably injective, we would like to present a classical proof of Theorem~\ref{Prop: c0I is separably injective} without passing through the isometric identification with $C_0(U_\mathcal I)$. 

\begin{theorem}[Sobczyk's Theorem for ideals]\label{thm:sobczyk}
Let $Y$ be a closed subspace of Banach space $X$ with $X/Y$ separable. Every bounded operator $T:Y\longrightarrow c_{0,\mathcal I}$ admits an extension $\widetilde T:X\longrightarrow c_{0,\mathcal I}$ such that $\|\widetilde T\|\leq2\|T\|$. In particular, $c_{0,\mathcal I}$ is $2$-separably injective.
\end{theorem}

\begin{proof}
Denote $T(y)=\bigl(y_n^*(y)\bigr)_{n\in\omega}$, for $y\in Y$, where $y_n^*\in Y^*$, $\|y_n^*\|\leq \|T\|$ with $\mathcal I\text{-}\lim_n y_n^*(y)=0$, for each $y\in Y$ and $M=\|T\|$.

By the Hahn-Banach extension theorem, for each $n\in\omega$ there is an extension $z_n^*\in X^*$ with $z_n^*|_Y=y_n^*$, and $\|z_n^*\|\leq M$. Since $X/Y$ is separable, let $(x_j)_{j\geq 1}\subseteq X$ such that $\spn\{x_j+Y:j\geq1\}$ is dense in $X/Y$. Define $C\doteq\{u\in X^*:u|_Y=0,\ \|u\|\leq M\}$ and, for each $m\geq1$:
\[
D_m\doteq\left\{n\in\omega:
\inf_{u\in C}\max_{1\leq j\leq m}
|(z_n^*-u)(x_j)|\geq\frac1m
\right\}.
\]

\begin{claim}
    $D_m\in\mathcal I$.

Suppose towards a contradiction that $D_m\notin \mathcal I$. Then there exists an ultrafilter $p\in\beta\omega$ such that $\mathcal I^*\cup\{D_m\}\subseteq p$. By the weak*-compactness of $B_{X^*}$ we may choose $z^*=w^*\text{-}\lim_p z_n^*$.

Observe that for every $y\in Y$, $z^*(y)=\lim_p y_n^*(y)=0,$ because $p\supseteq\mathcal I^*$ and $\mathcal I\text{-}\lim y_n^*(y)=0$. Therefore, $z^*\in C$. On the other hand, given the basic weak* open neighborhood $\mathcal U \doteq U(z^*;x_1,\ldots,x_m;1/m)$ and $A_m=\{n\in\omega:z_n^*\in\mathcal U \}$, one can show that $D_m\subseteq \N\setminus A_m$ with $A_m\in p$, a contradiction. Thus, $D_m\in\mathcal I$.
\end{claim}

For every $m\in\omega$, consider $E_m\doteq\bigcup_{k\leq m}D_k\in\mathcal I$ and define:
\[
r(n)\doteq \begin{cases}
    0, & \text{ if } \{m\leq n:n\notin E_m\} =\emptyset \\
    \max\{m\leq n:n\notin E_m\}, & \text{otherwise}.
\end{cases}
\] 
For each $k\in \omega$, we have that $\{n\in\omega\,\colon\,r(n)<k\}\subseteq[0,k)\cup E_k\in\mathcal I$. Indeed, if $n\geq k$ and $r(n)<k$, then $n\in E_k$.

If $r(n)>0$, $n\notin D_{r(n)}$ and we can choose $u_n^*\in C$ such that $\max_{1\leq j\leq r(n)} |(z_n^*-u_n^*)(x_j)|<\frac1{r(n)}.$ We define the functionals that will extend $T$ as follows:

\[
x_n^* = \begin{cases}
    z_n^*,& \text{if } r(n)=0 \\
    z_n^* -u_n^* ,& \text{if } r(n)>0.
\end{cases}
\]

\begin{claim}
    For each $j=1,\ldots,m$, it follows that $\mathcal I\text{-}\lim_n x_n^*(x_j)=0.$

Fix $j\in\{1,\ldots,m\}$ and let $\varepsilon>0$ be given. Choose $k\in\mathbb N$ such that $k>j$ and $1/k<\varepsilon.$ If $n\in\mathbb N$ and $r(n)\geq k$, since 
$\max_{1\leq i\leq r(n)}
|(z_n^*-u_n^*)(x_i)|<\frac1{r(n)}$ we have that $|x_n^*(x_j)|<1/r(n)\leq 1/k<\varepsilon.$ Consequently, $\{n\in\mathbb N\,\colon\,|x_n^*(x_j)|\geq\varepsilon\}\subseteq\{n\in\mathbb N\,\colon\,r(n)<k\}\in\mathcal I$ and the claim follows.    
\end{claim}

Notice that, by the definition of $T$, we also have $x_n^*(y)=y_n^*(y) \xrightarrow{\,\mathcal I\,}0$.

\begin{claim}
   For each $x\in X$ we have that $\mathcal I$-$\lim x_n^*(x)=0.$

Given $x\in X$ and $\varepsilon>0$, by density there exists a linear combination such that
\[\left\| x+Y-\sum_{j=1}^m \alpha_j(x_j+Y)\right\|_{X/Y}<\dfrac{\varepsilon}{6\|T\|}.\]

Choose $y\in Y$ such that $\left\| x-\sum_{j=1}^m \alpha_ix_i+y\right\|<\dfrac{\varepsilon}{6\|T\|}$. Observe that:
\begingroup
\allowdisplaybreaks
\begin{align*}
    |x_n^*(x_j)|&= \left|x_n^*\left(  x-  \sum_{j=1}^m \alpha_jx_j  \right) + 
    x_n^*\left(\sum_{j=1}^m \alpha_jx_j \right)\right| \\
    &\leq \left|x_n^*\left(  x-  \sum_{j=1}^m \alpha_jx_j +y \right) - y_n^*(y)  \right| 
        +\left| x_n^*\left(\sum_{j=1}^m \alpha_jx_j \right) \right| \\
    &\leq 2\|T\| \left\| x-  \sum_{j=1}^m \alpha_jx_j +y \right\| + |y_n^*(y)| +
        \left| x_n^*\left(\sum_{j=1}^m \alpha_jx_j \right) \right| \\
    &< \frac{\varepsilon}{3} + |y_n^*(y)| +
        \left| x_n^*\left(\sum_{j=1}^m \alpha_jx_j \right) \right|.
\end{align*}
\endgroup

The above inequality implies that
\begin{equation*}
    \{n\in\omega\,\colon\,|x_n^*(x)|\geq\varepsilon\}\subseteq\{n\in\omega\,\colon\,|y_n^*(y)|\geq\varepsilon/3\}\cup\left\{n\in\omega\,\colon\,\left|x_n^*\left(\sum_{j=1}^m \alpha_jx_j \right)\right|\geq\varepsilon/3\right\}.
\end{equation*}
Therefore, $\mathcal I\text{-}\lim_n x_n^*(x)=0.$ 
\end{claim}

Thus, $\widetilde T: X\rightarrow c_{0,\mathcal I}$ given by $\widetilde T(x)=\bigl(x_n^*(x)\bigr)_{n\in\omega}$ extends $T$ and $\|\widetilde T\|\leq 2\|T\|$.
\end{proof}

In what follows, $Y\leq X$ means: $Y$ is a closed subspace of the Banach space $X$.

\begin{definition}\label{def: comp and si constant}
Given a Banach space $X$, we define the \emph{constant of separable-injectivity} of $X$:
\[
\lambda_{\mathrm{si}}(X) \doteq
\inf\left\{ \lambda>0 \,\,\colon\,\,
\begin{aligned}
&\forall F\leq E,\ E \text{ separable},\  \forall T\in B(F,X),\\
&\exists \widetilde T\in B(E,X),\ \widetilde T|_F=T,\ 
    \|\widetilde T\|\leq\lambda\|T\|
\end{aligned}
\right\}.
\]

Fixed $Y\leq X$, the \emph{complementation constant of $Y$ in $X$} is given by:
\[
\lambda_{{\rm comp}}(Y,X) \doteq \inf\{ \|P\|:P:X\to Y\text{ is a projection}\},
\]
where the infima are taken in the extended line $\overline{\mathbb{R}}$.

When $X=\ell_\infty$, we write  $\lambda_{{\rm comp}}(Y,\ell_\infty)=\lambda_{{\rm comp}}(Y)$.
\end{definition}

The reader may notice the following alternative formula to compute the constant of complementation:
\[
\lambda_{{\rm comp}}(Y,X) = \inf\left\{ \|I_X- S\circ q\| \,\,\colon\,\, S\in B(X/Y, X),\  q\circ S=I_{X/Y}\right\},
\]
where $q:X\to X/Y$ is the quotient map. 

In the terms of Definition~\ref{def: comp and si constant}, Sobczyk's Theorem~\ref{thm:sobczyk} can be simply rephrased as: $\lambda_{\mathrm{si}}(c_{0,\mathcal I})\leq 2$.
For a proper ideal $\mathcal I$, the inclusion $c_0\subseteq c$ gives the reverse inequality, as it shows the next result:

\begin{theorem}[Sharpness of the Sobczyk constant]
\label{thm:sharp-sobczyk}
For every proper ideal $\mathcal I$, $ \lambda_{\mathrm{si}}(c_{0,\mathcal I})=2$.
\end{theorem}

\begin{proof}
It remains to prove that $\lambda_{\mathrm{si}}(c_{0,\mathcal I})\geq 2$. Let $c$ be the space of convergent sequences, let $L:c\to\K$ be the functional given by $L\big((x_n)\big) = \lim_{n\to\infty} x_n$ and $i:c_0\to c_{0,\mathcal I}$ be the inclusion operator. Suppose that $\widetilde T:c\to c_{0,\mathcal I}$
extends $i$, and set $z=\widetilde T(\mathds{1})$, where $\mathds{1}=(1,1,\ldots)$.  For every $x\in c$, 
\[
 \widetilde T(x)=\widetilde T\big( \underbrace{x-L(x) \mathds{1}}_{\in c_0}+L(x)\mathds{1}\big)=x-L(x)\mathds{1}+L(x)z.
\]
Consequently, for every $n$, $\pi_n \widetilde T=\pi_n+(z_n-1)L$ and it follows that:
\[
 \|\widetilde  T\|
 =\sup_n\|\pi_n \widetilde T\|
 =\sup_n\|\pi_n+(z_n-1)L\|
 =\sup_n\bigg(1+|1-z_n|\bigg)
 =1+\|\mathds{1}-z\|_\infty.
\]
Moreover, we will show that $ d\doteq \dist(\mathds{1},c_{0,\mathcal I})=1$. Indeed, to see that $d\leq 1$, take $z=0$.  Suppose towards a contradiction that $d<1$ and choose  $z\in c_{0,\mathcal I}$ such that $\|\mathds{1}-z\|_\infty<1$. Let $\alpha$ be such that $\|\mathds{1}-z\|_\infty<\alpha<1$. It follows that $|z_n|\geq \varepsilon\doteq1-\alpha>0$ for every $n$, thus, $\omega=\{n:|z_n|\geq\varepsilon\}\in\mathcal I$, a contradiction. Thus $\|\widetilde T\|\geq 2$. 
\end{proof}

We end this section with a few results on the complementation constants of the spaces $c_{0,\mathcal I}$. The following proposition concerns projections onto proper closed subspaces of $\ell_\infty$
containing $c_0$. We include a proof for completeness, as we have not found an
appropriate reference in the literature.

\begin{proposition}[Projection lower bound and splitting formula]
\label{prop:projection-formula}
Suppose that $Y$ is a proper closed subspace of $\ell_\infty$ containing $c_0$ and that $P:\ell_\infty\to Y$ is a projection. Let $Q=I_X-P$ and $q:X\to X/Y$ the quotient map. Then $\|P\|=1+\|Q\|\ \geq2$ and
\[
 \lambda_{\mathrm{comp}}(Y) =1+\inf\left\{ \|S\|: S:\ell_\infty/ Y\to\ell_\infty, \ q\circ S=I_{\ell_\infty/ Y} \right\}.
\]
\end{proposition}

\begin{proof}
Since $Q$ annihilates $Y\supseteq c_0$, one has
$Q^*(\pi_n)\in c_0^\perp$ for all $n\in\N$.  We claim that
\[
 \forall\,n\in\N:\quad\|P^*(\pi_n)\|  =\|\pi_n-Q^*(\pi_n)\|  =1+\|Q^*(\pi_n)\|.
\]
Indeed, the inequality ``$\leq$'' follows from the triangle inequality. For the reverse inequality, fix $\varepsilon>0$ and choose $x\in B_{\ell_\infty}$ such that, after replacing $x$ by $-x$ if necessary, $-\big(Q^*(\pi_n)\big)(x)>\|Q^*(\pi_n)\|-\varepsilon.$ Define $z=x+(1-x_n)e_n.$ Then $\|z\|_\infty\leq 1$ and, since $e_n\in c_0$ and $Q^*(\pi_n)\in c_0^\perp$,  we have $\big(Q^*(\pi_n)\big)(z)=\big(Q^*(\pi_n)\big)(x)$ and $\pi_n(z)=1.$ Hence,
\begin{equation*}
    \|\pi_n-Q^*(\pi_n)\|\geq|\big(\pi_n-Q^*(\pi_n)\big)(z)|=1-\big(Q^*(\pi_n)\big)(x)>1+\|Q^*(\pi_n)\|-\varepsilon.
\end{equation*}
Since $\varepsilon>0$ is arbitrary, $\|\pi_n-Q^*(\pi_n)\|
\geq 1+\|Q^*(\pi_n)\|,$ which implies the desired conclusion.

Taking suprema over $n$ yields $\|P\|=1+\|Q\|$ and, because $Y \subsetneq \ell_\infty$, $Q$ is a nonzero projection and $\|Q\|\geq1$.

Now we will prove the formula for $\lambda_{\text{comp}}(Y)$. 
Let $P\colon\ell_\infty\to Y$ be a projection and $Q= I_{\ell_\infty}-P$. Define $S\colon\ell_\infty/Y\to\ell_\infty$ by $S(x+Y)=Q(x)$, for each $x\in\ell_\infty$ and notice that $S$ is well defined, $Q=S\circ q$ and $\|S\|=\|Q\|=\|P\|-1$. Since $q\circ Q=q$ and $S\circ q=Q$, it follows that $q\circ S=I_{\ell_\infty/Y}$.
Consequently,
\[
\lambda_{\text{comp}}(Y)\geq1+\inf\left\{
 \|S\|:
 S:\ell_\infty/ Y \to\ell_\infty,
 \ q\circ S=I_{\ell_\infty/ Y}
 \right\}.
\]
Conversely, if $S:\ell_\infty/ Y \to\ell_\infty$ is a bounded operator such that $q\circ S=I_{\ell_\infty/Y}$, then $P= I_{\ell_\infty}-S\circ q$ is a projection from $\ell_\infty$ onto $Y$. Notice that 
$\|P\|=1+\|S\circ q\|=1+\|S\|$. Therefore,
\[\lambda_{\text{comp}}(Y)\leq 1+\inf\left\{
 \|S\|:
 S:\ell_\infty/ Y\to\ell_\infty,
 \ q\circ S=I_{\ell_\infty/ Y}
 \right\},\]
 and the proof is complete.
\end{proof}

As an immediate consequence of the previous result and the fact that $c_0\subseteq c_{0,\mathcal I} \subseteq \ell_\infty$, we obtain the following:

\begin{corollary}\label{cor:no-below-two}
No proper ideal has complementation constant or separable-injectivity
constant strictly smaller than $2$.
\end{corollary}

\begin{remark} The results stated in Proposition~\ref{prop:projection-formula} are not valid in general for complemented spaces. Consider $X=\ell_1^3$, $\mathds{1} = (1,1,1)$, $Y=\spn\{\mathds{1}\}$ and $P_0:\ell_1^3\to\ell_1^3$ given by
\[
P_0(x)= \left(\frac{x_1+x_2+x_3}{3}\right)\mathds 1.
\]
Observe that $P_0$ is a projection onto $Y$, $\|P_0\|=1$ and $\|I-P_0\|=\frac43$. Consequently, $\|P_0\|\neq 1+\|I-P_0\|$.

Furthermore, if $q:X\to X/Y$ denotes the quotient map, then:
\[
\lambda_{\mathrm{comp}}(Y,X)=1,
\qquad
\inf\{\|S\|:q\circ S=I_{X/Y}\}=\frac43.
\]
\end{remark}

\begin{remark} The only case in which the complementation and separable-injectivity constants are
below 2 occurs when $\mathcal I = \mathcal P(\omega)$ since $c_{0,\mathcal I}=\ell_\infty$, and both constants are $1$.  Thus the
separable-injectivity constant has the complete description given by:
\[
 \lambda_{\mathrm{si}}(c_{0,\mathcal I})
 =
 \begin{cases}
 1,&\mathcal I=\mathcal P(\omega),\\
 2,&\fin\subseteq\mathcal I\subsetneq\mathcal P(\omega).
 \end{cases}
\]
\end{remark}
For the complementation constant, the remaining open question is the following:
\begin{question}
    Does every complemented ideal has complementation constant exactly 2 or constants strictly larger than $2$ may occur?
\end{question}

By Proposition~\ref{prop:projection-formula}, the previous question is equivalent to
asking whether the canonical quotient map always has a contractive linear
right inverse whenever it has a bounded one.

\section{On statistical ideals and complementation of $c_{0,\mathcal I}$}
\label{section complementation}

In \cite{kadetsetal2022}, the authors introduced and studied a class of filters associated with positive measures defined on $2^\omega$. As shown in \cite{rincon-uzcateguiII}, this class of filters is closely connected to the problem of complementability of an ideal $\I$. We describe this connection below.  

According to \cite{kadetsetal2022}, a finitely additive positive measure $\mu$ defined in $2^\omega$ is called a \textit{statistical measure} if $\mu(\omega)=1$ and $\mu(\{k\})=0$ for all $k\in\omega$. The \textit{filter generated by a statistical measure $\mu$} is defined by
\[\mathfrak F_\mu\doteq\{A\subseteq\omega\,\colon\,\mu(A)=1\}.\]
An ideal $\I$ is called a \textit{statistical ideal} if there exists a statistical measure $\mu$ such that $\I^*=\mathfrak F_\mu$. Equivalently, a statistical ideal is an ideal of the form
\[\I_\mu=\{A\subseteq\omega\,\colon\,\mu(A)=0\},\]
where $\mu$ is a statistical measure. 

Before continuing, we need to recall some notions. A Boolean algebra $\mathbb B$ admits a strictly positive finitely additive measure if there is a finitely additive measure $\mu:\mathbb B\to\R^+$ such that for all $b\in\mathbb B\setminus\{{\bf0}\}$, we have $\mu(b)>0$. A compact space $K$ is \emph{approximable} if the closed unit ball of
$M(K)=C(K)^*$ is separable for the weak-star topology $\sigma(M(K),C(K))$.

\begin{theorem}\cite[Theorem~0.1]{HrusakSaenz}\label{complemented-approximable}
 Let $\I$ be an ideal. The following statements are equivalent:
 \begin{enumerate}[(1)]
     \item $\I$ is complemented;
     \item $\ell_\infty/c_{0,\I}$ is isomorphic to a closed subspace of $\ell_\infty$.
     \item $K_{\mathcal I}\cong\Stone(\mathcal P(\omega)/\mathcal I)$ is approximable.
 \end{enumerate}
 In particular, if $\I$ is complemented, then the Boolean algebra $\mathcal P(\omega)/\I$ admits a strictly positive finitely additive measure. 
\end{theorem}

\begin{corollary}\label{complemented implies statistical}
      If $\I$ is complemented, then $\I$ is a statistical ideal. 
\end{corollary}

\begin{proof}
    By Theorem~\ref{complemented-approximable}, $\mathcal P(\omega)/\mathcal I$ admits a strictly positive measure $m$. Define $\mu:\mathcal P(\omega)\to[0,1]$ by $\mu(A)=m([A]_\mathcal I)/m([\omega]_\mathcal I)$, for all $A\subseteq\omega$.  Then, $\mu$ is a statistical measure such that $\mathcal I=\mathcal I_\mu$.
\end{proof}
 
So, it is natural to ask if the converse of Corollary \ref{complemented implies statistical} holds. In this section we show that this is not the case. More precisely, we will prove that for a class of $\mathcal U\in\beta\omega\setminus\omega$, the statistical ideal
\[\mathcal I_\mathcal U=\left\{A\subseteq\omega:\mathcal U-\lim_{n}\frac{|A\cap[0,n)|}{n}=0\right\},\]
is not complemented. It is worth noting that this non-complemented ideal has a simpler description than the one shown in \cite[Example 2]{HrusakSaenz}. 

\subsection{A non-complemented statistical ideal}

For $A\subseteq\omega$ and $\mathcal U\in\beta\omega$, define $\displaystyle d_{\mathcal U}(A)
 =\mathcal U-\lim_{n}\dfrac{|A\cap[0,n)|}{n}.$ It is easy to check that $d_{\mathcal U}$ is a statistical measure. We also define
\begin{equation*}
 \Psi_{\mathcal U}(x)
 =\mathcal U-\lim_{n}\frac1n\sum_{k<n}x_k,
 \qquad x=(x_k)\in\ell_\infty.
\end{equation*}
The functional $\Psi_{\mathcal U}$ is positive, $\Psi_{\mathcal U}( \mathds{1})=1$, and therefore $\|\Psi_{\mathcal U}\|=1$.  Notice also that
$ \Psi_{\mathcal U}(\chi_A)=d_{\mathcal U}(A)$ for all $A\subseteq\omega.$

Recall that $\ell_\infty$ can be identified with $C(\beta\omega)$ by sending a sequence $x\in\ell_\infty$ to its unique Stone--\v{C}ech extension $\tilde x\in C(\beta\omega)$. By the Riesz representation theorem, $\Psi_\mathcal U$ corresponds to a unique Radon probability measure $\widetilde d_{\mathcal U}$ on $\beta\omega$. Under this identification, we have $\widetilde d_{\mathcal U}(\widehat A)=d_\mathcal U(A)$ for all $A\subseteq\omega$.

\medskip

Let $ \mathcal I_{\mathcal U}
 =\{A\subseteq\omega:d_{\mathcal U}(A)=0\}$ and $ \mathbb B_{\mathcal U}
 =\mathcal P(\omega)/\mathcal I_{\mathcal U}.$  Notice that every element of the family of clopens
$\{\widehat A:A\in\mathcal I_{\mathcal U}\}$ is $\widetilde d_\mathcal U$-null. The measure $d_{\mathcal U}$ defines a strictly positive, finitely
additive probability measure on $\mathbb B_{\mathcal U}$, namely, $\overline d_{\mathcal U}([A]_{\mathcal I_{\mathcal U}})
 =d_{\mathcal U}(A)$ for each $[A]_{\mathcal I_{\mathcal U}}\in \mathbb B_{\mathcal U}.$ This is well defined because $|d_{\mathcal U}(A)-d_{\mathcal U}(B)| \leq d_{\mathcal U}(A\mathbin\triangle B)$ for all $A,B\subseteq\omega$.
The measure $\overline d_{\mathcal U}$ is strictly positive by the definition of $\I_\mathcal U$. 

\begin{lemma}
\label{lem:support-stone}
One has
\begin{equation}
 \supp\widetilde d_{\mathcal U}
 =K_{\mathcal I_{\mathcal U}}
 \cong\Stone(\mathbb B_{\mathcal U}).
 \label{eq:support-stone}
\end{equation}
The homeomorphism in \eqref{eq:support-stone} is the canonical one coming
from Stone duality.
\end{lemma}

\begin{proof}
Recall that a point $p$ belongs to the support of a positive Radon measure if and only if every open neighborhood of $p$ has positive
measure. Suppose first that $p\notin\supp\widetilde d_{\mathcal U}$.  There is an open set $O\subseteq\beta\omega$ such that $ p\in O$ and 
 $\widetilde d_{\mathcal U}(O)=0.$ Choose $A\subseteq\omega$ with $ p\in\widehat A\subseteq O.$  Then, 
\[
 0\leq d_{\mathcal U}(A)
 =\widetilde d_{\mathcal U}(\widehat A)
 \leq\widetilde d_{\mathcal U}(O)=0.
\]
Hence $A\in\mathcal I_{\mathcal U}$.  Since $p\in\widehat A$ means exactly that $A\in p$, the ultrafilter $p$
meets $\mathcal I_{\mathcal U}$.  Thus, $p\notin K_{\mathcal I_{\mathcal U}}$.

Conversely, suppose that
$p\notin K_{\mathcal I_{\mathcal U}}$.  By definition, there is
$A\in p\cap\mathcal I_{\mathcal U}$.  Then $p\in\widehat A$, and so
 $ \widetilde d_{\mathcal U}(\widehat A)
 =d_{\mathcal U}(A)=0.$ Thus, $p$ has a null open neighborhood and cannot belong to the support.
This proves the first equality in \eqref{eq:support-stone}.

For completeness, we also describe the Stone-duality identification.
For $p\in K_{\mathcal I_{\mathcal U}}$, define $ \Theta(p)
 =\{[A]_{\mathcal I_{\mathcal U}}:A\in p\}.$ Because $p$ contains no member of $\mathcal I_{\mathcal U}$, this is a
well-defined ultrafilter on the quotient Boolean algebra
$\mathbb B_{\mathcal U}$.  Every ultrafilter on
$\mathbb B_{\mathcal U}$ arises in this way, and $ \Theta^{-1}
 \bigl(\{q:[A]_{\mathcal I_{\mathcal U}}\in q\}\bigr)
 =K_{\mathcal I_{\mathcal U}}\cap\widehat A.$

It follows that $\Theta$ is the canonical homeomorphism from
$K_{\mathcal I_{\mathcal U}}$ onto
$\Stone(\mathbb B_{\mathcal U})$.
\end{proof}

\begin{lemma}
\label{lem:images}
Let $K$ and $L$ be compact Hausdorff spaces.  If $K$ is approximable and
$\pi:K\to L$ is a continuous surjection, then $L$ is approximable.
\end{lemma}

\begin{proof}
Define $C_\pi:C(L)\longrightarrow C(K)$, $C_\pi f=f\circ\pi.$ Since $\pi$ is onto,
\[
 \|C_\pi f\|_{C(K)}
 =\sup_{x\in K}|f(\pi(x))|
 =\sup_{y\in L}|f(y)|
 =\|f\|_{C(L)}.
\]
Thus $C_\pi$ is an isometric embedding.  Its adjoint map $ C_\pi^*:\mu\in M(K)\longmapsto \mu\circ C_\pi\in M(L)$ is weak* continuous. We claim that $C_\pi^*$ maps the unit ball of $M(K)$ onto the unit ball
of $M(L)$.  Let $\nu\in M(L)$ be such that $\|\nu\|\leq1$ and $\Lambda\colon C_\pi(C(L))\to\mathbb R$ be given by
\[
 \Lambda(C_\pi f)=\int_L f\,d\nu,\quad f\in C(L).
\]
Since $C_\pi$ is an isometry, $\|\Lambda\|=\|\nu\|\leq1$. Let $\mu\in M(K)$ be a Hahn--Banach extension of $\Lambda$.  Then,
$\|\mu\|\leq1$ and $C_\pi^*\mu=\nu$, proving the claim.

Now let $D$ be a countable weak* dense subset of the unit ball of $M(K)$.  The set $C_\pi^*(D)$ is countable and continuity, together with surjectivity on the unit balls, shows that it is weak* dense in the unit ball of $M(L)$.  Hence, $L$ is approximable.
\end{proof}

 For $n\geq1$, write $\underline n=\{0,1,\ldots,n-1\}$ and equip $\mathcal P(\underline n)$ with normalized counting measure
$m_n(F)=\frac{|F|}{n},$ $F\in\mathcal P(\underline n)$. On the direct product $\displaystyle\prod_{n\geq1}\mathcal P(\underline n)$ define the pseudometric
\begin{equation*}
 \rho_{\mathcal V}(\mathbf F,\mathbf G)
 =\mathcal V-\lim_n
   \frac{|F_n\mathbin\triangle G_n|}{n},
 \qquad
 \mathbf F=(F_n),\quad \mathbf G=(G_n),
\end{equation*}
where $\mathcal V$ is a free ultrafilter.  The null ideal is $\mathcal N_{\mathcal V}
 =\left\{
   \mathbf F=(F_n):
   \mathcal V-\lim_n\frac{|F_n|}{n}=0
  \right\}.$

The metric ultraproduct is the quotient Boolean algebra
\begin{equation*}
 \mathbb A_{\mathcal V}
 =\left(\prod_{n\geq1}\mathcal P(\underline n)\right)
  \big/\mathcal N_{\mathcal V}.
\end{equation*}
An element of $\mathbb A_{\mathcal V}$ will be written $[\mathbf F]_{\mathcal N_{\mathcal V}}
 =[(F_n)_{n\geq1}]_{\mathcal N_{\mathcal V}}.$ The quotient carries the strictly positive probability measure
\begin{equation*}
 m_{\mathcal V}
 \bigl([\mathbf F]_{\mathcal N_{\mathcal V}}\bigr)
 =\mathcal V-\lim_n\frac{|F_n|}{n}.
\end{equation*}
The metric is $ \rho_{\mathcal V}(a,b)=m_{\mathcal V}(a\mathbin\triangle b)$ for all $a,b\in\mathbb A_\mathcal V$.  The quotient is complete for this metric, Boolean operations are
continuous, and it is a complete probability algebra.  Equivalently,
countable joins can be obtained as metric limits of finite joins.  These
standard facts about measure ultraproducts are discussed, for example, in
\cite{Fremlin}.

There is a canonical map from the density quotient into this
ultraproduct:
\begin{equation}
 e_{\mathcal V}:\mathbb B_{\mathcal V}\longrightarrow
 \mathbb A_{\mathcal V},
 \qquad
 e_{\mathcal V}([A]_{\mathcal I_{\mathcal V}})
 =[(A\cap\underline n)_{n\geq1}]_{\mathcal N_{\mathcal V}}.
 \label{eq:canonical-e}
\end{equation}
This map is well defined. Indeed, if $A,B\subseteq\omega$, then
\begin{equation*}
     \rho_{\mathcal V}
 \bigl((A\cap\underline n)_n,(B\cap\underline n)_n\bigr)
 =\mathcal V-\lim_n
   \frac{|(A\cap\underline n)\mathbin\triangle
                (B\cap\underline n)|}{n}
 =\mathcal V-\lim_n
   \frac{|(A\mathbin\triangle B)\cap\underline n|}{n}
 =d_{\mathcal V}(A\mathbin\triangle B).
\end{equation*}
Consequently,
$[A]_{\mathcal I_{\mathcal V}}=[B]_{\mathcal I_{\mathcal V}}$
if and only if the two sequences on the right of
\eqref{eq:canonical-e} represent the same ultraproduct element.
Thus, $e_{\mathcal V}$ is well defined and injective.  Observe that $e_{\mathcal V}$ preserves all
finite Boolean operations coordinatewise, and
\begin{equation*}
 m_{\mathcal V}
 \bigl(e_{\mathcal V}([A]_{\mathcal I_{\mathcal V}})\bigr)
 =d_{\mathcal V}(A),\quad\text{ for all }A\subseteq\omega.
\end{equation*}
Hence $e_{\mathcal V}$ is an isometric, measure-preserving Boolean
embedding.

We now state exactly the consequences of Greb\'ik's construction that are
used in this section. Recall that an ultrafilter $\mathcal V$ is \emph{thin} if
\begin{equation*}
 \inf_{B\in\mathcal V}
 \limsup_{j\to\infty}\frac{b_j}{b_{j+1}}=0,\quad\text{where}\quad B=\{b_0<b_1<\cdots\}.
\end{equation*}

Now we introduce the combinatorial property that controls countable
additivity.  For an integer $k>1$ and a set $A\subseteq\omega$, put
\begin{equation*}
 \Dil_k(A)
 =\bigcup_{n\in A}[kn,(k+1)n].
\end{equation*}

\begin{definition}
A free ultrafilter $\mathcal U$ is \emph{$\times$-invariant} if for every
$A\in\mathcal U$ there exists an integer $k>1$ such that
$\Dil_k(A)\in\mathcal U$.
\end{definition}

The upcoming result will be used in the proof of the main theorem in this section.

\begin{theorem}\cite[Theorem~1.5]{Grebik}\label{eq:grebik-dichotomy}
    $\overline d_{\mathcal U}$ is countably additive on $\mathbb B_{\mathcal U}$ if and only if $ \mathcal U$ is not $\times$-invariant.
\end{theorem}

\begin{theorem}[Greb\'ik]
\label{thm:grebik-reduction}
Let $\mathcal U$ be a free ultrafilter on $\omega$.
\begin{enumerate}[(G1)]
\item \cite[Theorem 2.5 and Corollary 2.9]{Grebik} There exists a free ultrafilter
  $\mathcal V=G(\mathcal U)$ which is thin and satisfies $d_{\mathcal V}=d_{\mathcal U}.$  Moreover, $\mathcal V$ is $\times$-invariant if and only if
  $\mathcal U$ is $\times$-invariant.
\item \cite[Proposition 2.10]{Grebik} If $\mathcal V$ is thin, then the embedding
  $e_{\mathcal V}$ in \eqref{eq:canonical-e} has dense range, and it is
  onto if and only if $\overline d_{\mathcal V}$ is countably additive
  on $\mathbb B_{\mathcal V}$.
\end{enumerate}
\end{theorem}

\begin{lemma}\label{eq:B-is-A}
   If  $\mathcal U$ is not $\times$-invariant and $\mathcal V=G(\mathcal U)$, then $\mathbb B_{\mathcal U} =\mathbb B_{\mathcal V} \cong\mathbb A_{\mathcal V}$.
\end{lemma}

\begin{proof}
    By (G1) of Theorem \ref{thm:grebik-reduction}, $\mathcal V$ is thin and not $\times$-invariant. It follows from Theorem \ref{eq:grebik-dichotomy} that
  $\overline d_{\mathcal V}$ is countably additive. By (G2) of Theorem \ref{thm:grebik-reduction}, $e_{\mathcal V}$ is onto. Since $d_{\mathcal V}=d_{\mathcal U}$, one has
  $\mathcal I_{\mathcal V}=\mathcal I_{\mathcal U}$ and hence
  $\mathbb B_{\mathcal U}=\mathbb B_{\mathcal V}\cong\mathbb A_{\mathcal V}$.
\end{proof}

Now we will identify the probability algebra $\mathbb A_{\mathcal V}$. For a probability algebra $\mathbb A$, its \textit{Maharam type} is the density character of the
metric space $(\mathbb A,\rho)$, where $\rho(a,b)=m(a\mathbin\triangle b)$. Also, $\mathbb A$ is \emph{homogeneous} if every nonzero principal algebra $ \mathbb A\mathbin\upharpoonright a =\{b\in\mathbb A:b\leq a\}$ has the same Maharam type as $\mathbb A$.  Let $\Balg(\kappa)$ denote the standard homogeneous probability algebra of Maharam type $\kappa$, equivalently the measure algebra of the product probability space $2^\kappa$.

\begin{lemma}
\label{lem:ultraproduct-type}
For every free ultrafilter $\mathcal V$, the probability algebra $\mathbb A_{\mathcal V}$ is atomless and homogeneous of Maharam type $\cardc$.  Consequently, $\mathbb A_{\mathcal V}\cong\Balg(\cardc)$ as probability algebras.
\end{lemma}

\begin{proof}
We divide the proof into four steps.

\smallskip
\noindent
\emph{Step 1: the upper bound.}
Every element of $\mathbb A_{\mathcal V}$ is represented by a sequence
$(F_n)$ with $F_n\subseteq\underline n$.  Therefore
\[
 |\mathbb A_{\mathcal V}|
 \leq\left|\prod_{n\geq1}\mathcal P(\underline n)\right|
 \leq\aleph_0^{\aleph_0}=\cardc.
\]
In particular, the metric density of every subspace of
$\mathbb A_{\mathcal V}$ is at most $\cardc$.

\smallskip
\noindent
\emph{Step 2: a separated family below every nonzero element.}
Let $a\in\mathbb A_{\mathcal V}$ be nonzero.  Choose a representative $a=[\mathbf F]_{\mathcal N_{\mathcal V}}$ with $\mathbf F=(F_n)_{n\geq1}$ and $F_n\subseteq\underline n$ for each $n\in\omega$. Put
$ \alpha=m_{\mathcal V}(a)
 =\displaystyle\mathcal V-\lim_n\frac{|F_n|}{n}>0.$ The definition of $\mathcal V$-limit gives that $ L_0=\left\{n\in\omega:\frac{|F_n|}{n}\geq\frac\alpha2\right\}\in\mathcal V$.

For $n\in L_0$, let $N_n$ be the largest power of two not exceeding
$|F_n|$.  Write $N_n=2^{q_n}$, choose
$H_n\subseteq F_n$ with $|H_n|=N_n$, and fix a bijection $\theta_n:H_n\longrightarrow\mathbb F_2^{q_n}.$ For indices outside $L_0$, put $H_n=\varnothing$ and $q_n=0$.
 Let $q\geq 1$ be given. Then, $A_0=\{n\in\omega:n>\frac{4}{\alpha}q\}\in\mathcal V$. On the other hand, if $n\in L_0\cap A_0$, then $N_n>|F_n|/2$, and so $N_n>\dfrac{|F_n|}{2}\geq\dfrac{\alpha n}{4}>q$.
 It follows that $\{n\in\omega:2^{q_n}>q\}\in\mathcal V$. Hence, $q_n\to\infty$ along $\mathcal V$: for every $q\in\omega$, $\{n\in\omega:q_n>q\}\in\mathcal V$.

For each binary sequence $r\in2^\omega$, define a sequence of finite
sets $\mathbf E_r=(E_r(n))_{n\geq1}$ by
\begin{equation}
 E_r(n)
 =\left\{
   h\in H_n:
   \left\langle r\mathbin\upharpoonright q_n,\theta_n(h)\right\rangle=1
  \right\},\quad\text{where}\quad  \langle u,v\rangle
 =\sum_{j<q_n}u(j)v(j)\pmod2
 \label{eq:Er-definition}
\end{equation}
is the standard bilinear form on $\mathbb F_2^{q_n}$.  Put $e_r=[\mathbf E_r]_{\mathcal N_{\mathcal V}} \in\mathbb A_{\mathcal V}.$ Since $E_r(n)\subseteq H_n\subseteq F_n$ for every $n\in\omega$, we have
$e_r\leq a$.

Let $r,s\in2^\omega$ be distinct, and choose $j_0$ with
$r(j_0)\neq s(j_0)$.  Since $q_n\to\infty$ along $\mathcal V$, $L_{r,s}=\{n\in\omega:q_n>j_0\}\in\mathcal V$. Fix $n\in L_{r,s}\cap L_0$ and put $ u=r\mathbin\upharpoonright q_n$ and $ v=s\mathbin\upharpoonright q_n.$ Then $u\neq v$, so $u+v\neq0$ in $\mathbb F_2^{q_n}$.  From \eqref{eq:Er-definition},
\begin{align*}
       E_r(n)\mathbin\triangle E_s(n)&=\{ h\in H_n: \langle u,\theta_n(h)\rangle \neq\langle v,\theta_n(h)\rangle \} \\
 &=\{ h\in H_n: \langle u+v,\theta_n(h)\rangle=1\}\\
 &=\theta_n^{-1}(\{z\in\mathbb F_2^{q_n}:\langle u+v,z\rangle=1\}).
\end{align*}
The nonzero linear functional $ w\in \mathbb F_2^{q_n}\longmapsto\langle u+v,w\rangle\in\mathbb F_2$ has a kernel of cardinality $2^{q_n-1}$. Its other fiber has the samecardinality.  Consequently,
\begin{equation*}
 |E_r(n)\mathbin\triangle E_s(n)|
 =2^{q_n-1}
 =\frac{N_n}{2}
 \geq\frac{|F_n|}{4}, \quad \text{for all }n\in L_{r,s}\cap L_0. 
\end{equation*}
It follows from the above equation by taking the $\mathcal V$-limit and using the definition of $\alpha$ that
\begin{equation}
 \rho_{\mathcal V}(e_r,e_s)=m_{\mathcal V}(e_r\mathbin\triangle e_s)=\mathcal V-\lim_n\frac{|E_r(n)\triangle E_s(n)|}{n}\geq\mathcal V-\lim_n\frac{|F_n|}{4n}=\frac\alpha4.
 \label{eq:Er-separation}
\end{equation}
Thus $ \{e_r:r\in2^\omega\}
 \subseteq\mathbb A_{\mathcal V}\mathbin\upharpoonright a$ is a uniformly separated family of cardinality $\cardc$.

\smallskip
\noindent
\emph{Step 3: homogeneity and atomlessness.}
It is a standard fact that in a metric space, a $\delta$-separated family has cardinality at most
the density character: each ball of radius $\delta/3$ around a member of
the family must contain a different point of any dense set.  Therefore
\eqref{eq:Er-separation} implies $\dens(\mathbb A_{\mathcal V}\mathbin\upharpoonright a)
 \geq\cardc.$ Together with Step~1, this yields $\dens(\mathbb A_{\mathcal V}\mathbin\upharpoonright a)
 =\cardc$. Hence the algebra is homogeneous of Maharam type $\cardc$.

To see atomlessness directly, start again with a nonzero
$a=[(F_n)]_{\mathcal N_{\mathcal V}}$ of measure $\alpha$.  For each $n\in\omega$ take a partition $F_n=F_n^0\mathbin{\dot\cup}F_n^1,$ where $\bigl||F_n^0|-|F_n^1|\bigr|\leq1.$ Then, $\displaystyle\left|\frac{|F_n^i|}{n}-\frac{|F_n|}{2n} \right|\leq\frac1{2n}$ for $i=0,1$. Since $\mathcal V$ is free, $\mathcal V-\lim_{n}1/n=0$, and therefore $m_{\mathcal V}([(F_n^0)])=m_{\mathcal V}([(F_n^1)])=\frac\alpha2>0.$ Both classes are nonzero proper elements below $a$, so $a$ is not an
atom.

\smallskip
\noindent
\emph{Step 4: Maharam's classification.}
We have proved that $\mathbb A_{\mathcal V}$ is a complete atomless
homogeneous probability algebra of Maharam type $\cardc$.  Maharam's
classification theorem \cite{Fremlin} now gives the measure-preserving
Boolean isomorphism
$\mathbb A_{\mathcal V}\cong\Balg(\cardc)$.
\end{proof}

  A consequence of the M\"agerl--Namioka intersection-number criterion is
\begin{equation}
 \Stone(\Balg(\omega_1))
 \text{ is not approximable}.
 \label{eq:Bomega1-not-approximable}
\end{equation}

See \cite{MagerlNamioka} and the discussion accompanying the
approximability criterion in \cite{HrusakSaenz}.

We are ready to prove the main theorem of this section. 

\begin{theorem}
\label{thm:nonapprox-support}
Let $\mathcal U$ be a free ultrafilter which is not
$\times$-invariant.  Then, $ \supp\widetilde d_{\mathcal U}
 \cong\Stone(\mathbb B_{\mathcal U})$ is not approximable.  Consequently, $ c_{0,\mathcal I_{\mathcal U}}$ is not complemented in $\ell_\infty.$
\end{theorem}

\begin{proof}
Set $\mathcal V=G(\mathcal U)$ as in Theorem~\ref{thm:grebik-reduction}.  Because $\mathcal U$ is not $\times$-invariant, Lemma \ref{eq:B-is-A} implies that $ \mathbb B_{\mathcal U} =\mathbb B_{\mathcal V}  \cong\mathbb A_{\mathcal V}.$ By Lemma~\ref{lem:ultraproduct-type} we obtain $ \mathbb B_{\mathcal U}\cong\Balg(\cardc).$ Since $\omega_1\leq\cardc$, the measure algebra on the first $\omega_1$ coordinates is naturally a Boolean subalgebra $ \Balg(\omega_1)\hookrightarrow\Balg(\cardc).$ By the Stone duality, the restriction map 
\[
 \rho:\Stone(\Balg(\cardc))\longrightarrow
       \Stone(\Balg(\omega_1)),
 \qquad
 \rho(p)=p\cap\Balg(\omega_1)
\]
defines a continuous map. Since every ultrafilter on the subalgebra extends to an ultrafilter on the larger Boolean algebra, $\rho$ is onto.

If $\Stone(\mathbb B_{\mathcal U})$ were approximable, then $\Stone(\Balg(\cardc))$ would be approximable. Since the map $\rho:\Stone(\Balg(\cardc))\rightarrow\Stone(\Balg(\omega_1))$ is a surjection,
by Lemma~\ref{lem:images} we conclude that $\Stone(\Balg(\omega_1))$ is approximable, which contradicts \eqref{eq:Bomega1-not-approximable}.  Thus,
$\Stone(\mathbb B_{\mathcal U})$ is not approximable. Therefore, $K_{\mathcal I_{\mathcal U}}$ is not approximable by Lemma~\ref{lem:support-stone}. Now, the result follows from Theorem \ref{complemented-approximable}.
\end{proof}

The next corollary provides an explicit ZFC class of ultrafilters to which
Theorem~\ref{thm:nonapprox-support} applies. 

\begin{corollary}
\label{cor:lacunary-U}
Let $T=\{t_0<t_1<\cdots\}\subseteq\omega$ satisfy $\displaystyle\frac{t_j}{t_{j+1}}\longrightarrow0.$ If $\mathcal U$ is any free ultrafilter containing $T$, then
$\supp\widetilde d_{\mathcal U}$ is not approximable and
$c_{0,\mathcal I_{\mathcal U}}$ is not complemented in
$\ell_\infty$.
\end{corollary}

\begin{proof}
We verify directly that $\mathcal U$ is not $\times$-invariant.  Fix an
integer $k>1$.  By hypothesis, there is $j_0$ such that
\begin{equation}
 t_{j+1}>(k+1)t_j
 \qquad(j\geq j_0).
 \label{eq:lacunary-gap}
\end{equation}
We claim that $T\cap\Dil_k(T)$ is finite.  Indeed, suppose that $j\geq j_0$ and that some $t_m$ belongs to the interval
$[kt_j,(k+1)t_j]$.  Since $k>1$, we have $t_m>t_j$, so $m\geq j+1$. But then \eqref{eq:lacunary-gap} gives $ t_m\geq t_{j+1}>(k+1)t_j,$ contradicting $t_m\in[kt_j,(k+1)t_j]$.  Hence, only the finitely many intervals corresponding to $j<j_0$ can meet $T$, which proves the claim.

Now $T\in\mathcal U$.  If $\Dil_k(T)$ also belonged to $\mathcal U$,
then their intersection would belong to $\mathcal U$.  This is impossible
because $\mathcal U$ is free and $T\cap\Dil_k(T)$ is finite.  Therefore, $\Dil_k(T)\notin\mathcal U$ for every $k>1.$ Thus, $T\in\mathcal U$ witnesses that $\mathcal U$ is not $\times$-invariant. The conclusion follows from Theorem~\ref{thm:nonapprox-support}.
\end{proof}

For a concrete choice, take $t_j=2^{2^j}.$ Then $t_j/t_{j+1}=2^{-2^j}\to0$.  The filter generated by $T$ and the
cofinite subsets of $\omega$ is free and proper, so it extends to a free
ultrafilter.  Every such extension satisfies
Corollary~\ref{cor:lacunary-U}.

\begin{remark}
\label{rem:times-invariant-open}
If $\mathcal U$ is $\times$-invariant, then Theorem \ref{eq:grebik-dichotomy} says that $\overline d_{\mathcal U}$ is not countably additive on $\mathbb B_{\mathcal U}$.  The ultrafilter $\mathcal V=G(\mathcal U)$ is thin and has the same density as $\mathcal U$. By (G2) of Theorem \ref{thm:grebik-reduction}, the canonical embedding $ e_{\mathcal V}:\mathbb B_{\mathcal V}\longrightarrow  \mathbb A_{\mathcal V}$ is not onto.  Consequently, the identification $\mathbb B_{\mathcal U}\cong\Balg(\cardc)$ used in Theorem~\ref{thm:nonapprox-support} does not hold.  The approximability conclusion remains open for the $\times$-invariant case.
\end{remark}

\section{Complemented copies of \texorpdfstring{$c_0$}{c0} in \texorpdfstring{$c_{0,\mathcal I}(X)$}{c0I(X)}}
\label{section complemented copies of c0}

In this section we study the complementation of $c_0$ in the Banach space $c_{0,\I}(X)$. In the setting of $\I=\mathcal P(\omega)$, a result of \cite{leung-rabiger} states that if $X$ is an infinite-dimensional Banach space, then $c_0\xhookrightarrow{c}\ell_\infty(X)=c_{0,\I}(X)$ if and only if $c_0\xhookrightarrow{c} X$. So, it is natural to ask:

\begin{question}\label{question complementation}
  Let $X$ be an infinite-dimensional Banach space. For which ideals $\I$ do we have
  \begin{equation}\label{equation question complementation}
      c_0\xhookrightarrow{c} c_{0,\mathcal I}(X)\quad\Longleftrightarrow\quad c_0\xhookrightarrow{c} X\ ?
  \end{equation}
\end{question}

It is straightforward that, if either $c_0 \xhookrightarrow{c} c_{0,\I}$ or $c_0 \xhookrightarrow{c} X$, then $c_0 \xhookrightarrow{c} c_{0,\mathcal I}(X)$ because $c_{0,\I}$ and $X$ are complemented in $c_{0,\mathcal I}(X)$. 
On the other hand, not every ideal satisfies the statement \eqref{equation question complementation}. Indeed, the ideal $\I=\fin$ does not satisfy the statement \eqref{equation question complementation}: if $X$ is an infinite-dimensional Banach space, then $c_{0,\mathcal I}\stackrel{\vee}{\otimes} X$ contains a complemented copy of $c_0$ by the main theorem of \cite{cembranos} (recall that if $X$ and $Y$ are Banach spaces, $X\stackrel{\vee}{\otimes} Y$ denotes the \textit{injective tensor product} of $X$ and $Y$, see \cite[Chapter 1]{diestel-and-all}). Also, we know that $c_{0,\mathcal I}\stackrel{\vee}{\otimes} X\equiv c_{0,\I}(X)$ by \cite[Theorem 1.1.11]{diestel-and-all}. Thus, $c_0\xhookrightarrow{c}c_{0,\I}(X)$. Moreover, the following result holds.

\begin{proposition}\label{equivalent complementation}
    Let $\mathcal I$ be an ideal on $\mathbb N$ and $X$ be a Banach space. Then, $c_0\xhookrightarrow{c} c_{0,\mathcal I}(X)$ if and only if
    $c_0\xhookrightarrow{c} c_{0,\mathcal I^\omega}(X)$.
\end{proposition}

\begin{proof}
    Suppose that $c_0\xhookrightarrow{c} c_{0,\mathcal I}(X)$. Then, $\ell_\infty(c_0)\xhookrightarrow{c} \ell_\infty(c_{0,\mathcal I}(X))\cong c_{0,\mathcal I^\omega}(X)$
    by \cite[Theorem 5.2]{rincon-uzcategui}. Since $c_0\xhookrightarrow{c}\ell_\infty(c_0)$, we conclude that  $c_0\xhookrightarrow{c} c_{0,\mathcal I^\omega}(X)$.
    
    Conversely, assume that $c_0\xhookrightarrow{c} c_{0,\mathcal I^\omega}(X)$. Then, $c_0\xhookrightarrow{c} \ell_\infty(c_{0,\mathcal I}(X))$. It follows from
    \cite[Main Theorem]{leung-rabiger} that $c_0\xhookrightarrow{c} c_{0,\mathcal I}(X)$.
\end{proof}

Thus, if $X$ is a Banach space such that $c_0\not\xhookrightarrow{c}X$ and $\mathcal I=\fin^\omega$, since $c_0\xhookrightarrow{c}c_0(X)=c_{0,\fin}(X)$, by Proposition \ref{equivalent complementation} we obtain $c_0\xhookrightarrow{c}c_{0,\mathcal I}(X)$.  Also, if $\mathcal I$ is an ideal on $\omega$ and $X$ is a Banach space, then $c_0\xhookrightarrow{c} c_{0,\mathcal I^{\omega\perp}}(X)$. Indeed, $c_{0,\mathcal I^{\omega\perp}}(X)\equiv c_0(c_{0,\mathcal I^\perp}(X))\equiv c_0 \stackrel{\vee}{\otimes} c_{0,\mathcal I^\perp}(X)$ by \cite[Theorem 5.4]{rincon-uzcategui} and \cite[Theorem 1.1.11]{diestel-and-all}. The result follows from the main theorem of \cite{cembranos}.

Therefore, we need an additional notion that allows us to establish the equivalence we seek. According to \cite{hrusák2026grothendieckidealsellinfty}, an ideal $\I$ on $\mathbb N$ is called \textit{Grothendieck} if $c_{0,\mathcal I}$ has the Grothendieck property, that is, every sequence of elements of $c_{0,\I}^*$ which is weak$^*$-convergent is weak-convergent. One of the results of \cite{hrusák2026grothendieckidealsellinfty} establishes that
\begin{equation*}
    \I \text{ is Grothendieck }\quad\Longleftrightarrow\quad c_0\not\xhookrightarrow{c} c_{0,\I}.
\end{equation*}

Before we provide an answer to the Question \ref{question complementation}, we need to recall the celebrated Schlumprecht's characterization of complemented copies of $c_0$: 

\begin{theorem}\cite{schlumprecht1987limitierte}\label{schlumprecht theorem}
    Let $X$ be a Banach space. Then, $c_0\xhookrightarrow{c} X$ if and only if there are a sequence $(x_n)$ in $X$ equivalent to the canonical basis of $c_0$ and a weak$^*$-null sequence of functionals $(x_n^*)$ in $X^*$ such that $\inf_{n\in\mathbb N}|x_n^*(x_n)|>0.$
\end{theorem}

The next result answers Question \ref{question complementation}.

\begin{theorem}\label{thm-grothendieck}
Let $\mathcal I$ be an ideal on $\omega$ and let $X$ be a Banach space.  Then,
\[
 c_0\xhookrightarrow{c} c_{0,\mathcal I}(X)
 \quad\Longleftrightarrow\quad c_0\xhookrightarrow{c} c_{0,\I}\ \text{ or }\
 c_0\xhookrightarrow{c} X.
\]
Consequently, if $\mathcal I$ is Grothendieck, then 
\[
 c_0\xhookrightarrow{c} c_{0,\mathcal I}(X)
 \quad\Longleftrightarrow\quad
 c_0\xhookrightarrow{c} X.
\]
\end{theorem}

\begin{proof}

By the discussion above it only remains to prove one direction.

$(\Rightarrow)$  Put $Y=c_{0,\mathcal I}(X)$ and suppose that $\I$ is Grothendieck and $c_0\xhookrightarrow{c} Y$.  There are bounded
operators $J:c_0\longrightarrow Y$ and $ R:Y\longrightarrow c_0$ such that
$ RJ=I_{c_0}. $ If $(e_n)$ is the canonical basis of $c_0$, define $z_n=Je_n\in Y$ and $f_n=e_n^*R\in Y^*$ for each $n\in\mathbb N$. Then, $\sup_n\|z_n\|<\infty$, $\sup_n\|f_n\|<\infty$ and $ f_n(z_m)=\delta_{nm}$ for each $m,n\in\mathbb N$.  Moreover, for each $y\in Y$, $f_n(y)=(Ry)_n\longrightarrow0$. Thus, $f_n\xrightarrow{w^*}0$ in $Y^*.$ For $A\subseteq\N$ and $n\in\mathbb N$, set $ \lambda_n(A)=f_n(z_n\chi_A).$ Since  $\{z\chi_A:z\in Y\}\subseteq Y$, this is
well defined for every subset $A$ of $\N$.  Observe that each $\lambda_n$ is a bounded finitely additive scalar measure on $\mathcal P(\N)$.

\begin{claim}\label{lem:variation}
For every $n\in\mathbb N$ and every $A\subseteq\N$, $|\lambda_n|(A)
 \leq \|f_n\|\,\|z_n\chi_A\|_\infty$. Thus, $(\lambda_n)$ is bounded in $\ba(\N)=\ell_\infty^*$.

Indeed, let $A=A_1\mathbin{\dot\cup}\cdots\mathbin{\dot\cup}A_m$ be a finite
partition.  Choose signs, or complex phases, $\theta_j$ with
$|\theta_j|=1$ so that
$\theta_j\lambda_n(A_j)=|\lambda_n(A_j)|$.  Then
\[
 \sum_{j=1}^m|\lambda_n(A_j)|
 =f_n\!\left(\sum_{j=1}^m\theta_jz_n\chi_{A_j}\right)
 \leq \|f_n\|\,\|z_n\chi_A\|_\infty.
\]
Taking the supremum over finite partitions proves the claim.
\end{claim}

\begin{claim}
    The sequence $(\lambda_n)$ is relatively weakly compact in $\ba(\N)$.

   If not, by the Dieudonn\'e--Grothendieck criterion for bounded subsets of $\ba(\N)$ \cite[Theorem 5.3.2]{AK}, after passage to subsequences, there are a sequence of pairwise
disjoint subsets $(A_j)$ of $\N$ and $\varepsilon>0$ such that $ |\lambda_{n_j}|(A_j)>\varepsilon$ for all $j\in\mathbb N$. By the definition of total variation, for each $j\in\N$ choose scalar simple functions $u_j$ such that ${\rm supp}(u_j)\subseteq A_j$, $\|u_j\|_\infty\leq1$ and $|f_{n_j}(y_j)|>\varepsilon.$ For all $j,k\in\mathbb N$, set $y_j=z_{n_j}u_j$ and $h_j(k)=\|y_j(k)\|$. Then, $h_j\in c_{0,\mathcal I}$, ${\rm supp}(h_j)\subseteq A_j$ and
\[
 0<\frac{\varepsilon}{\sup_n\|f_n\|}
 \leq \|h_j\|_\infty
 \leq \sup_n\|z_n\|,
\]
for all $j\in\mathbb N$. Consequently, $(h_j)$ is equivalent to the canonical basis of $c_0$.

Put $v(k)=\dfrac{y_j(k)}{\|y_j(k)\|}$ whenever $k\in\supp(y_j)$ for some $j\in\mathbb N$, and $v(k)=0$ otherwise. Notice that the operator 
\[V:c_{0,\mathcal I}\longrightarrow Y,\quad (Va)(k)=a(k)v(k), \ a\in c_{0,\I},\ k\in\mathbb N,\]
is well defined and continuous, since $\|(Va)(k)\|\leq|a(k)|$ for all $a\in c_{0,\I}$ and $k\in\mathbb N$, and it satisfies
$Vh_j=y_j$ for every $j\in\mathbb N$.  Let $\varphi_j=V^*f_{n_j}\in c_{0,\mathcal I}^*$ for $j\in\mathbb N$. Since $(f_{n_j})$ is weak$^*$-null, $(\varphi_j)$ is weak$^*$-null.  Moreover, $|\varphi_j(h_j)|=|f_{n_j}(y_j)|>\varepsilon$ for each $j\in\mathbb N$. By Theorem~\ref{schlumprecht theorem}, $c_0\xhookrightarrow{c}c_{0,\mathcal I}$, contrary to the hypothesis.  Therefore
$(\lambda_n)$ is relatively weakly compact.
\end{claim}

Pass to a subsequence with $\lambda_n\to\lambda$ weakly in $\ba(\N)$. Notice that $\lambda(\N)=1$.  

\begin{claim}
    $\lambda(A)\neq0$ for some $A\in\mathcal I$.

Suppose that $\lambda(A)=0$ for every $A\in\mathcal I$. Observe that $|\lambda|(A)=0$ for each $A\in\mathcal I$. We also denote by $\lambda$ the Radon measure on $\beta\N$ representing $\lambda$. Notice that $|\lambda|(\widehat{B})=|\lambda|(B)$ for all $B\subseteq\N$. Let $F\subset U_\mathcal I$ be compact. Since $U_\mathcal I=\bigcup_{A\in\mathcal I}\widehat A$, there are $A_1,\ldots,A_n\in\mathcal I$ such that
$F\subseteq\bigcup_{1\leq j\leq n}\widehat{A_j}$. Consequently, $|\lambda|(F)\leq\sum_{j=1}^n|\lambda|(A_j)=0.$ The arbitrariness of $F$ implies that $|\lambda|(U_\mathcal I)=0$. Hence, $\lambda$ is supported on $K_\mathcal I=\beta\N\setminus U_\mathcal I$ and $ \lambda(K_\mathcal I)=1.$

Choose $\eta>0$ with $M\eta<1/2$, where
$M=\sup_n\|f_n\|$, and set for each $n\in\N:$
\[
 D_n\doteq\{k\in\N:\|z_n(k)\|\geq\eta\}\in\mathcal I,
 \qquad \text{and}\qquad
 C_n\doteq\N\setminus D_n\in\mathcal I^*.
\]
Since $K_\mathcal I\subseteq\widehat{C_n}$,  by Claim~\ref{lem:variation}, we have:
\[|\lambda_n(K_\mathcal I)|\leq |\lambda_n|(\widehat{C_n})=|\lambda_n|(C_n)<1/2,\text{ for all }n\in\N.\]

On the other hand,  the equality $\lambda(K_\mathcal I)=1$ implies that $\lambda_n(K_\mathcal I)\to1$, a contradiction.
\end{claim}

To conclude, fix $A\in\mathcal I$ with $\lambda(A)\neq0$.  After deleting finitely many terms, there is $\delta>0$ such that
\begin{equation}\label{eq:positive-on-A}
    |f_n(z_n\chi_A)|\geq\delta,\quad\text{for all }n\in\mathbb N.
\end{equation}
  Let $Y_A=\{y\in Y:\supp y\subseteq A\}.$ Because $A\in\mathcal I$, every bounded $X$-valued sequence supported on $A$ is
$\mathcal I$-null.  Hence, $ Y_A\cong\ell_\infty(A,X)$ via the map $y=(y_n)\mapsto (y_n)_{n\in A}$. The operator $ S:c_0\longrightarrow Y_A$ defined by $ Se_n=z_n\chi_A$ for each $n\in\mathbb N$ is bounded.  Equation \eqref{eq:positive-on-A} gives $\inf_n\|Se_n\|\geq \dfrac{\delta}{\sup_n\|f_n\|}>0.$ The sequence $(Se_n)$ is weakly null and has an upper $c_0$ estimate. The Bessaga--Pe\l czy\'nski selection principle yields a subsequence $(Se_{n_k})$ equivalent to the canonical basis of $c_0$. Let $g_k=f_{n_k}|_{Y_A}$.  Then $(g_k)$ is weak-star null in $Y_A^*$ and $|g_k(Se_{n_k})|\geq\delta$ for all $k\in\N$. By Theorem \ref{schlumprecht theorem},
$Y_A$ contains a complemented copy of $c_0$. The set $A$ is nonempty, since $\lambda(A)\neq0$. Since $Y_A$ is the $\ell_\infty$-sum of countably many copies of $X$, it follows from the main theorem of \cite{leung-rabiger} that $c_0\xhookrightarrow{c} X$.
\end{proof}

\section{On \texorpdfstring{$\mathcal I$}{I}-AD families and the quotient \texorpdfstring{$\ell_\infty/c_{0,\mathcal I}$}{linfty/c0I}}\label{section quotiens}

In \cite{rincon-uzcateguiII}, the authors introduced a cardinal invariant 
associated with a family of $\mathcal I$-positive sets in order to study the complementation of $\I$. 
Recall that a family $\mathcal{A}\subseteq \I^+$ is  $\I$-AD if $X\cap Y\in \I$ for all distinct  elements $X,Y\in \mathcal A$. According to \cite{rincon-uzcateguiII}, let
\begin{gather*}
\mathfrak{ad}_\I(\mathcal A)=\sup\{|\mathcal B|\,\colon\,\mbox{$\mathcal B\subseteq \mathcal A$ is an $\I$-AD family}\}.
\end{gather*}
We also write $\mathfrak{ad}(\I)\doteq\mathfrak{ad}_\I(\I^+)$ and we omit the index when the ideal is clear by context.

As shown by Kania in \cite{kania}, the complementation of $c_{0,\mathcal I}$ in $\ell_\infty$ is related to the existence of uncountable $\I$-AD families. More precisely, Kania proved that if $\mathfrak{ad}(\I)$ is uncountable, then $\ell_\infty/c_{0,\mathcal I}$ contains an isometric copy of $c_0(\mathcal A)$ for some uncountable $\mathcal A$ (see the definition of $c_0(\mathcal A)$ below). Our goal in this section is to refine Kania's result. To this end, we start by studying the connection between $\mathcal I$-AD families and the quotient $\ell_\infty/c_{0,\mathcal I}$. We will analyze the isometric structure of the Banach space $\ell_\infty(X)/c_{0,\mathcal I}(X)$. 
We conclude by computing \(\mathfrak{ad}\) for Fubini products and giving an example in which interchanging the factors changes its value.

\subsection{Some structural results about \texorpdfstring{$\ell_\infty/c_{0,\mathcal I}$}{linfty/c0I}}

We start by recalling some notation. Let $\mathcal I$ be an ideal on a countably infinite set $F$. For each sequence ${\bf x}=(x_n)_{n\in F}$  in $\mathbb R$, we set $B_{\bf x}(\mathcal I)=\{b\in\mathbb R\,\colon\,\{k\in F\,\colon\,x_k>b\}\not\in\mathcal I\}$ and 
\begin{gather*}
 \mathcal I-\limsup{\bf x}=
 \begin{cases}
\sup B_{\bf x}(\mathcal I),& B_{\bf x}(\mathcal I)\neq\emptyset;\\
-\infty, & B_{\bf x}(\mathcal I)=\emptyset.
 \end{cases}
 \end{gather*}
 Notice that if ${\bf x}\in\ell_\infty(F)$ and $\mathcal I$ is proper, we have $B_{\bf x}(\mathcal I)\neq\emptyset$ and $-\infty<\displaystyle\mathcal I-\limsup{\bf x}=\sup B_{\bf x}(\mathcal I)<\infty$. If ${\bf x}=(x_{(m,n)})\in\ell_\infty(\mathbb N\times\mathbb N)(X)$, we set ${\bf x}_{(m)}=(x_{(m,n)})_{n\in\mathbb N}$.

 \begin{remark}\label{norm in the quotient}
     If $\mathcal I$ is an ideal on a countably infinite set $F$ and ${\bf x}=(x_n)_{n\in F}$ is a bounded sequence in a Banach space $X$, then the proof of \cite[Lemma 5.7]{rincon-uzcategui} shows that
     \begin{equation*}
         \|{\bf x}+c_{0,\mathcal I}(X)\|=\mathcal I-\limsup_n\|x_n\|.
     \end{equation*}
 \end{remark}

\begin{lemma}
\label{superior limit of sums}
Let $(\mathcal I_n)_{n\in\mathbb N}$ be a sequence of proper ideals on $\mathbb N$ and ${\bf x}=(x_{(m,n)})\in\ell_\infty(\N\times\N)$. Then,
    \begin{gather*}
       \left( \bigoplus_{n\in\mathbb N}\mathcal I_n\right)-\limsup_{(m,n)} x_{(m,n)}=\sup_{m\in\mathbb N}(\mathcal I_m-\limsup_n x_{(m,n)}).
    \end{gather*}
\end{lemma}

\begin{proof}
Let $\mathcal I\coloneqq\bigoplus_{n\in\mathbb N}\mathcal I_n$ and $A\in\mathbb R$ be such that $\sup_{m\in\mathbb N}\,(\mathcal I_m-\limsup_{n} x_{(m,n)})> A$. We have
  \begin{align*}
      \sup_{m\in\mathbb N}\,(\mathcal I_m-\limsup_{n} x_{(m,n)})> A&\Longleftrightarrow(\exists\,m\in\mathbb N)(\mathcal I_m-\limsup_n x_{(m,n)}> A)\\
&\Longleftrightarrow(\exists\,m\in\mathbb N)(\{n\in\mathbb N\,\colon\,x_{(m,n)}>A\}\not\in\mathcal I_m)\\
&\Longleftrightarrow\{(k,n)\in\mathbb N\times\mathbb N\,\colon\,x_{(k,n)}>A\}\not\in\mathcal I.
  \end{align*}
  Thus, $\mathcal I-\limsup_{(m,n)}x_{(m,n)}\geq A$. Since $A$ was arbitrary, we conclude that 
  \begin{gather}\label{first}
      \mathcal I-\limsup_{(m,n)}x_{(m,n)}\geq\sup_{m\in\mathbb N}\,(\mathcal I_m-\limsup_{n} x_{(m,n)}).
  \end{gather}
Now fix $\varepsilon>0$ such that $\mathcal I-\limsup_{(m,n)}x_{(m,n)}-\varepsilon<b$ where $b\in\mathbb R$ satisfies  $\{(k,n)\in\mathbb N\times\mathbb N\,\colon\,x_{(k,n)}>b\}\not\in\mathcal I$. So, there is $m\in\mathbb N$ such that $\{n\in\mathbb N\,\colon\,x_{(m,n)}>b\}\not\in\mathcal I_m$. Then, $\mathcal I_m-\limsup_n x_{(m,n)}\geq b$.  It follows that 
   \begin{gather}\label{second}
      \mathcal I-\limsup_{(m,n)}x_{(m,n)}\leq\sup_{m\in\mathbb N}\,(\mathcal I_m-\limsup_{n} x_{(m,n)}).
  \end{gather}
From \eqref{first} and \eqref{second} we obtain the result.
\end{proof}

\begin{theorem}
\label{quotient of a sum}
Let $X$ be a Banach lattice, $(\mathcal I_m)_{m\in\mathbb N}$ be a sequence of proper ideals on $\mathbb N$, and $\mathcal I\doteq\bigoplus_{m\in\mathbb N}\mathcal I_m$. Then,
\begin{gather*}
   \ell_\infty(X)/c_{0,\mathcal I}(X)\equiv_{\leq}\ell_\infty((\ell_\infty(X)/c_{0,\mathcal I_m}(X))_{m\in\mathbb N}).
\end{gather*}
\end{theorem}

\begin{proof}
Consider the following map:
\begin{align*}
\xi\colon\ell_\infty(X)&\longrightarrow\ell_\infty((\ell_\infty(X)/c_{0,\mathcal I_m}(X))_{m\in\mathbb N}).\\
    {\bf x}=(x_{(m,n)})&\longmapsto({\bf x}_{(m)}+c_{0,\mathcal I_m}(X))_{m\in\mathbb N}.
\end{align*}
Notice that $\xi$ is linear, continuous, preserves the lattice operations, and $\xi$ is onto because it is a composition of onto linear operators. Moreover:

\begin{align*}
\xi({\bf x})={\bf0}&\Longleftrightarrow(\forall\,m\in\mathbb N)({\bf x}_{(m)}\in c_{0,\mathcal I_m}(X))\\
    &\Longleftrightarrow(\forall\,m\in\mathbb N)(\forall\,\varepsilon>0)(A({\bf x}_{(m)},\varepsilon)\in \mathcal I_m)\\
        &\Longleftrightarrow(\forall\,\varepsilon>0)(\forall\,m\in\mathbb N)\bigl(A(\mathbf{x},\varepsilon)_{(m)}\in\mathcal I_m\bigr)\\
    &\Longleftrightarrow(\forall\,\varepsilon>0)(A({\bf x},\varepsilon)\in \mathcal I).
\end{align*}
Therefore, $\ker\xi=c_{0,\mathcal I}(X)$. Finally, by Remark \ref{norm in the quotient} and Lemma \ref{superior limit of sums} we have
\begin{align*}
    \|{\bf x}+\ker\xi\|&=\mathcal I-\limsup_{(m,n)}\|x_{(m,n)}\|\\
    &=\sup_{m\in\mathbb N}(\mathcal I_m-\limsup_{n}\|x_{(m,n)}\|)\\
    &=\sup_{m\in\mathbb N}\|{\bf x}_{(m)}+c_{0,\mathcal I_m}(X)\|\\
    &=\|\xi({\bf x})\|.
\end{align*}
The above equation says that the quotient map induced by $\xi$ is a Banach lattice isometry from $\ell_\infty(X)/c_{0,\mathcal I}(X)$ onto $\ell_\infty((\ell_\infty(X)/c_{0,\mathcal I_m}(X))_{m\in\mathbb N}).$
\end{proof}

\begin{lemma}
\label{superior limit of fubini products}
Let $\mathcal I$ be an ideal on $\mathbb N$ and $(\mathcal I_n)_{n\in\mathbb N}$ be a sequence of proper ideals on $\mathbb N$. If ${\bf x}=(x_{(m,n)})\in\ell_\infty(\N\times\N)$, then
\begin{gather*}
\left(\sum_{\mathcal I}\mathcal I_n\right)-\limsup_{(m,n)} x_{(m,n)}=\mathcal I-\limsup_m\,\bigg(\mathcal I_m-\limsup_{n} x_{(m,n)}\bigg).
\end{gather*}
\end{lemma}

\begin{proof}
Put $\mathcal J\doteq\sum_{\mathcal I}\mathcal I_n$ and for $b\in\mathbb R$, $A(b)\doteq\{(m,n)\in\mathbb N\times\mathbb N\,\colon\,x_{(m,n)}>b\}$. 

\begin{claim}\label{identity for frolik sums}
\begin{enumerate}[(1)]
    \item If $m\in\N$ satisfies $A(b)_{(m)}\not\in\I_m$, then $\I_m-\limsup_nx_{(m,n)}\geq b.$
    \item If $\I_m-\limsup_nx_{(m,n)}> b$, then $A(b)_{(m)}\not\in\I_m$. 
\end{enumerate}

\medskip

If $k\in\N$ satisfies $A(b)_{(k)}=\{n\in\N:x_{(k,n)}>b\}\not\in\I_k$, it follows from the definition of limsup that $\I_k-\limsup_n x_{(k,n)}\geq b$. Now, if 
  $\I_m-\limsup_nx_{(m,n)}> b$, there is $c\in\R$ with $c\geq b$ and $A(c)_{(m)}=\{n\in\N:x_{(m,n)}>c\}\not\in\I_m$. Thus, $A(b)_{(m)}\not\in\I_m$ since
  $A(c)_{(m)}\subseteq A(b)_{(m)}$.
\end{claim}

 Let $\varepsilon>0$ be given. There is $b\in\mathbb R$ such that $\mathcal J-\limsup_{(m,n)} x_{(m,n)}-\varepsilon<b$ and $A(b)\not\in\mathcal J$.
Thus, $\{m\in\N:A(b)_{(m)}\not\in\I_m\}\not\in\mathcal I$, which implies $\{m\in\mathbb N\,\colon\,\mathcal I_m-\limsup_n x_{(m,n)}\geq b\}\not\in\mathcal I$ by 
(1) of Claim \ref{identity for frolik sums}. Hence, $ b\leq \mathcal I-\limsup_m\,(\mathcal I_m-\limsup_{n} x_{(m,n)}).$ The arbitrariness of $\varepsilon>0$ yields that 
  \begin{gather*}
      \left(\sum_{\mathcal I}\mathcal I_n\right)-\limsup_{(m,n)} x_{(m,n)}\leq\mathcal I-\limsup_m\,\bigg(\mathcal I_m-\limsup_{n} x_{(m,n)}\bigg).
  \end{gather*}

  Conversely, let $\varepsilon>0$ and $b\in\mathbb R$ be such that
  $\mathcal I-\limsup_m\,(\mathcal I_m-\limsup_{n} x_{(m,n)})-\varepsilon<b-\varepsilon/2$ and $ \{m\in\mathbb N\,\colon\,\mathcal I_m-\limsup_n x_{(m,n)}\geq b\}\not\in\mathcal I$. Notice that $ \{m\in\mathbb N\,\colon\,\mathcal I_m-\limsup_n x_{(m,n)}> b-\varepsilon/2\}\not\in\mathcal I$.
   It follows from (2) of Claim \ref{identity for frolik sums} that $A(b-\varepsilon/2)\not\in\mathcal J$, that is, $b-\varepsilon/2\leq \mathcal J-\limsup_{(m,n)} x_{(m,n)}$. Therefore,
  \begin{gather*}
      \mathcal I-\limsup_m\,\bigg(\mathcal I_m-\limsup_{n} x_{(m,n)}\bigg)\leq \left(\sum_{\mathcal I}\mathcal I_n\right)-\limsup_{(m,n)} x_{(m,n)}.
  \end{gather*}
\end{proof}

In the following result, we analyze the quotient of $\ell_\infty$ by $c_{0,\sum_\mathcal I\mathcal I_n}$.

\begin{theorem}\label{isomorphism quotient of a fubini-product}
Let $\mathcal I$ be an ideal and $(\I_n)_{n\in\mathbb N}$ be a sequence of proper ideals on $\mathbb N$, and let $X$ be a Banach lattice. Then 
\begin{gather*}
   \ell_\infty(X)/c_{0,\sum_{\mathcal I}\mathcal I_n}(X)\equiv_{\leq}\ell_\infty((\ell_\infty(X)/c_{0,\I_n}(X))_n)\Big{/}c_{0,\mathcal I}((\ell_\infty(X)/c_{0,\mathcal I_n}(X))_n).
\end{gather*}
In particular, if $\I_n=\mathcal J$ for all $n\in\mathbb N$, then 
\begin{gather*}
   \ell_\infty(X)/c_{0,\mathcal I\times\mathcal J}(X)\equiv_{\leq}\ell_\infty(\ell_\infty(X)/c_{0,\mathcal J}(X))\Big{/}c_{0,\mathcal I}(\ell_\infty(X)/c_{0,\mathcal J}(X)).
\end{gather*}
\end{theorem}

\begin{proof}
Let 
\begin{align*}
    \pi\colon\ell_\infty(X)&\to\ell_\infty((\ell_\infty(X)/c_{0,\mathcal I_m}(X))_m)\Big{/}c_{0,\mathcal I}((\ell_\infty(X)/c_{0,\mathcal I_m}(X))_m)\\
    {\bf x}=(x_{(m,n)})&\mapsto y({\bf x})+c_{0,\mathcal I}((\ell_\infty(X)/c_{0,\mathcal I_m}(X))_m),
\end{align*}
where $y({\bf x})=({\bf x}_{(m)}+c_{0,\mathcal I_m}(X))_{m\in\mathbb N}$. Clearly, $\pi$ is linear, continuous and preserves the lattice operations. Also, $\pi$ is onto since it is a composition of onto linear operators. Notice that
\begin{align*}
    \pi({\bf x})={\bf0}&\Longleftrightarrow ({\bf x}_{(m)}+c_{0,\mathcal I_m}(X))_{m\in\mathbb N}\in c_{0,\mathcal I}((\ell_\infty(X)/c_{0,\mathcal I_m}(X))_m)\\
    &\Longleftrightarrow(\forall\,\varepsilon>0)(A(y({\bf x}),\varepsilon)\in\mathcal I)\\
    &\Longleftrightarrow(\forall\,\varepsilon>0)(\{m\in\mathbb N\,\colon\,\|{\bf x}_{(m)}+c_{0,\mathcal I_m}(X)\|\geq\varepsilon\}\in\mathcal I)\\
     &\Longleftrightarrow(\forall\,\varepsilon>0)(\{m\in\mathbb N\,\colon\,\mathcal I_m-\limsup_n\|x_{(m,n)}\|\geq\varepsilon\}\in\mathcal I)\\
      &\Longleftrightarrow(\forall\,\varepsilon>0)(\{m\in\mathbb N\,\colon\,A({\bf x}_{(m)},\varepsilon)\not\in\mathcal I_m\}\in\mathcal I)\\
     &\Longleftrightarrow(\forall\,\varepsilon>0)(\{m\in\mathbb N\,\colon\,A({\bf x},\varepsilon)_{(m)}\not\in\mathcal I_m\}\in\mathcal I)\\
     &\Longleftrightarrow(\forall\,\varepsilon>0)(A({\bf x},\varepsilon)\in\sum_{\mathcal I}\mathcal I_n).
\end{align*}
Thus, $\ker\pi=c_{0,\mathcal I\times(\mathcal J_m)_m}(X)$. Now, by Lemma \ref{superior limit of fubini products} and Remark \ref{norm in the quotient}, we have
\begin{align*}
     \|{\bf x}+\ker\pi\|=\|{\bf x}+c_{0,\sum_{\mathcal I}\mathcal I_n}(X)\|&=\left(\sum_{\mathcal I}\mathcal I_n\right)-\limsup_{(m,n)}\|x_{(m,n)}\|\\
    &=\mathcal I-\limsup_m\,\bigg(\mathcal I_m-\limsup_{n}\|x_{(m,n)}\|\bigg)\\
    &=\mathcal I-\limsup_m\|{\bf x}_{(m)}+c_{0,\mathcal I_m}(X)\|\\
    &=\|y({\bf x})+c_{0,\mathcal I}((\ell_\infty(X)/c_{0,\mathcal I_m}(X))_m)\|.
\end{align*}
So, the quotient map induced by $\pi$ is a Banach lattice isometry from $\ell_\infty(X)/c_{0,\sum_\I\I_n}(X)$ onto $\ell_\infty((\ell_\infty(X)/c_{0,\mathcal I_n}(X))_n)\Big{/}c_{0,\mathcal I}((\ell_\infty(X)/c_{0,\mathcal I_n}(X))_n)$.
\end{proof}

\subsection{The interplay between \texorpdfstring{$\ell_\infty/c_{0,\mathcal I}$}{linfty/c0I} and \texorpdfstring{$\I$}{I}-AD families}
 
 We begin by collecting some facts about $\mathfrak{ad}(\mathcal I)$.

\begin{remark}\label{facts about ad(I)}
\begin{enumerate}[(1)]
\item If $\mathcal I$ is meager, then $\mathfrak{ad}(\I)=2^{\aleph_0}$ \cite[Lemma 2.3]{Leonetti2018}. The converse is false as  there is a non-meager ideal $\mathcal I$ with a  $\mathcal I$-AD family of size
$2^{\aleph_0}$ (see  \cite[Example 5.1]{farkas-khomskii-vidnyanszky}).

\item If $c_{0,\I}$ is complemented in $\ell_\infty$, then $\mathfrak{ad}(\I)\leq\aleph_0$ \cite[Theorem A]{kania}. 
The converse fails by \cite[Example 1]{HrusakSaenz}.
\end{enumerate}
\end{remark}

For a compact Hausdorff space $K$, $C(K)$ denotes the Banach space of all continuous functions from $K$ to $\mathbb K$, endowed with the supremum norm.
If $\tau$ is a cardinal, $c_0(\tau)$ is the Banach space of all functions $f\colon\tau\to\mathbb K$ such that for each $\varepsilon>0$ the set $\{\gamma\in\tau\,\colon\,|f(\gamma)|\geq\varepsilon\}$ is finite.  

The following result is essentially due to H. Rosenthal. Indeed, (1) $\Leftrightarrow$ (2) is proved in \cite[p. 230]{rosenthal}. (3) $\Rightarrow$ (2) is clear. On the other hand,  (1) $\Rightarrow$ (3) follows by a standard application of Urysohn's Lemma (see for instance \cite[Proposition 4.3.11]{AK}).

\begin{theorem}\label{caracterizacion cc rosental}
    Let $K$ be a compact Hausdorff space and $\mathfrak m$ be a cardinal. The following statements are equivalent:
    \begin{enumerate}[(1)]
        \item $K$ does not satisfy the $\mathfrak m$-chain condition, that is, there is a disjoint family of nonempty open subsets of $K$ with cardinality $\mathfrak m$.
        \item There is an isomorphism from $c_0(\mathfrak m)$ into $C(K)$.
        \item There is a Banach lattice isomorphism from $c_0(\mathfrak m)$ into $C(K)$.
    \end{enumerate}
\end{theorem}

Now we are ready to prove the main theorem of this section. Recall that
the \textit{cellularity} of a topological space $Y$ is defined by:
\[
c(Y) = \sup\{|\mathcal F|: \mathcal F \text{ is a disjoint family of nonempty open sets}\}.
\]

Likewise, the cellularity for a Boolean algebra $\mathbb B$ is given by: 
\[
c(\mathbb B) = \sup\{|\mathcal F| : \mathcal F \text{ is a family of
pairwise disjoint nonzero elements of } \mathbb B\}.
\]

Recall also that in Section~\ref{preliminares} we defined the closed set
$K_\I=\{p\in\beta\N:\I^*\subseteq p\}\subseteq\beta\N$ associated with
the ideal $\I$. The next theorem shows that $K_\I$ and the quotient
Boolean algebra $\mathcal P(\N)/\I$ both have cellularity
$\mathfrak{ad}(\I)$. The formulation via $\I$-almost disjoint families,
which we adopt from \cite{rincon-uzcateguiII}, describes this common invariant
directly in terms of the ideal. In Theorem~8.11 we use this description
to compute the value of $\mathfrak{ad}$ on a Fubini product from its
values on the two factors.

\begin{theorem}
\label{AD-families and KI}
Let $\mathcal I$ be an ideal on $\mathbb N$ and  $\kappa$ be an infinite cardinal. 
\begin{enumerate}[(1)]
    \item There is a Banach lattice isomorphism from $c_0(\kappa)$ into $\ell_\infty/c_{0,\mathcal I}$ if and only if there is an $\I$-AD family of size $\kappa$.
    \item  $\mathfrak{ad}(\mathcal I)=c(K_{\mathcal I})=c(\mathcal P(\mathbb N)/\mathcal I)$.
\end{enumerate}
\end{theorem}

\begin{proof}
(1) Suppose that $\mathcal A=\{A_\alpha\}_{\alpha\in\kappa}$ is an $\mathcal I$-AD family of size $\kappa$. Define $\psi\colon\spn\{e_\alpha\,\colon\,\alpha\in\kappa\}\to\ell_\infty$ by
$$
\psi\left(\sum_{\alpha\in F}b_\alpha e_\alpha\right)=\sum_{\alpha\in F}b_\alpha\chi_{A_\alpha},\quad \sum_{\alpha\in F}b_\alpha e_\alpha\in \spn\{e_\alpha\,\colon\,\alpha\in\kappa\}.
$$
Notice that $\psi$ is linear. So, the operator $\phi\colon\spn\{e_\alpha\,\colon\,\alpha\in\kappa\}\to\ell_\infty/c_{0,\mathcal I}$ given by $\phi(x)=\psi(x)+c_{0,\mathcal I}$ if $x\in\mathrm{span}\{e_\alpha\,\colon\,\alpha\in\kappa\}$, is linear. To see that $\phi$ preserves the lattice operations, notice that $\{\chi_{A_\alpha}+c_{0,\mathcal I}\,\colon\,\alpha\in \kappa\}$ are disjoint positive elements of $\ell_\infty/c_{0,\mathcal I}$. Thus,
\begin{align*}
     \left|\phi\left(\sum_{\alpha\in F}b_\alpha e_\alpha\right)\right|=\left|\left(\sum_{\alpha\in F}b_\alpha\chi_{A_\alpha}\right)+c_{0,\mathcal I}\right|
     =\sum_{\alpha\in F}|b_\alpha|(\chi_{A_\alpha}+c_{0,\mathcal I})=\phi\left(\sum_{\alpha\in F}|b_\alpha| e_\alpha\right),
\end{align*}
if $F\subset\kappa$ is a finite set and $\{b_\alpha\,\colon\,\alpha\in F\}\subset\mathbb K$. Also, note that 
    \begin{gather*}
       \left \|\phi(\sum_{\alpha\in F}b_\alpha e_\alpha)\right\|=\left\|\sum_{\alpha\in F}b_\alpha\chi_{A_\alpha}+c_{0,\mathcal I}\right\|=\mathcal I-\limsup\left|\sum_{\alpha\in F}b_\alpha\chi_{A_\alpha}\right|=\max_{\alpha\in F}|b_\alpha|,
    \end{gather*}
for $F\subset\kappa$ finite and $\{b_\alpha\,\colon\,\alpha\in F\}\subset\mathbb K$. So, $\phi$ can be extended to a Banach lattice isometry from  $c_0(\kappa)$ into $\ell_\infty/c_{0,\mathcal I}$.
    
Conversely, let $T\colon c_0(\kappa)\to\ell_\infty/c_{0,\mathcal I}$ be a Banach lattice isomorphism. For each $\alpha\in\kappa$, choose $y_\alpha\in\ell_\infty^+$ such that $T(e_\alpha)=y_\alpha+c_{0,\mathcal I}$.
Since $y_\alpha\not\in c_{0,\mathcal I}$, there are a set $A_\alpha\not\in\mathcal I$ and $\varepsilon_\alpha>0$ such that $\varepsilon_\alpha\chi_{A_\alpha}\leq y_\alpha$. If $\alpha,\beta\in\kappa$ and $\alpha\neq\beta$, then     
\begin{gather*}        
{\bf0}\leq(\varepsilon_\alpha\chi_{A_\alpha}+c_{0,\mathcal I})\wedge(\varepsilon_\beta\chi_{A_\beta}+c_{0,\mathcal I})\leq (y_\alpha+c_{0,\mathcal I})\wedge (y_\beta+c_{0,\mathcal I})=T(e_\alpha\wedge e_\beta)={\bf0}.
\end{gather*}
So, $A_\alpha\cap A_\beta\in\mathcal I$. Hence, $\{A_\alpha\,\colon\,\alpha\in\kappa\}$ is an $\mathcal I$-AD family of size $\kappa$.

(2) Let $\mathbb B=\mathcal P(\mathbb N)/\mathcal I$. Since
$[A]_{\mathcal I}\wedge[B]_{\mathcal I}=0$ if and only if
$A\cap B\in\mathcal I$, we have $\mathfrak{ad}(\mathcal I)=c(\mathbb B)$.
As $K_{\mathcal I}$ is the Stone space of $\mathbb B$,
\cite[Theorem~514B(b)]{Fremlin5} gives $c(\mathbb B)=c(K_{\mathcal I})$.
\end{proof}

The following result is probably known. We include a proof based on Theorem \ref{AD-families and KI}.  The idea comes from \cite[Proof of Theorem A]{kania}.
Recall that $\ell_\infty(\tau)$ denotes the Banach space of all bounded functions $f\colon\tau\to\mathbb K$ endowed with the supremum norm. Also,  $\ell_1(\tau)$
is the Banach space of all functions $f\colon\tau\to\mathbb K$ such that $\sum_{a\in\tau}|f(a)|<\infty$.

\begin{corollary}
\label{omegamaximal}
Let $\mathcal I$ be an ideal on $\mathbb N$ and $\kappa$ be a cardinal. If $\mathcal I$ is $\kappa$-maximal, then every $\mathcal I$-AD family has size at most $\kappa$.
\end{corollary}

\begin{proof}
Let  $\{\mathcal I_\alpha\,\colon\,\alpha\in\kappa\}$ be a family of maximal ideals such that $\mathcal I=\bigcap_{\alpha\in\kappa}\mathcal I_\alpha$.
Then, the operator ${\bf x}\in\ell_\infty\mapsto\Phi({\bf x})=(\mathcal I_\alpha^*-\lim{\bf x})\in\ell_\infty(\kappa)$ is linear, continuous, and $\ker\Phi=c_{0,\mathcal I}$.

If $\kappa$ is finite, each member of an $\mathcal I$-AD family belongs to some $\mathcal I_\alpha^*$, and no such ultrafilter contains two distinct members of the family. Hence the family has size at most $\kappa$. If $\kappa$ is infinite, assume, towards a contradiction,  that there is an $\mathcal I$-AD family of size $\kappa^+$. By Theorem \ref{AD-families and KI}, we obtain an embedding  $\psi\colon c_0(\kappa^+)\to\ell_\infty(\kappa)$. Then
$\psi^*\colon\ell_\infty(\kappa)^*\to c_0(\kappa^+)^*$ is an onto linear operator. By Goldstine's theorem $\ell_\infty(\kappa)^*=\ell_1(\kappa)^{**}=\overline{\ell_1(\kappa)}^{\sigma(\ell_1(\kappa)^{**},\ell_1(\kappa)^*)}$, so $w^*$-$\mathrm{dens}(\ell_\infty(\kappa)^*)=\kappa$. On the other hand, $w^*$-$\mathrm{dens}(c_0(\kappa^+)^*)=\kappa^+$ \cite[Theorem 13.3]{fabian-et-al}. This contradicts the well-known fact that if $f\colon X\to Y$ is an onto continuous map between topological spaces, then $\mathrm{dens}(Y)\leq\mathrm{dens}(X)$.
\end{proof}

The previous result is  an equivalence when $\kappa$ is a finite cardinal. This is a consequence of \cite[Proposition 4.3 and Corollary 4.7]{rincon-uzcateguiII}.

\begin{proposition}\label{caracterizacion ad(I) finite}
Let $\mathcal I$ be an ideal on $\mathbb N$ and $k$ a positive integer.  Then $\mathfrak{ad}(\mathcal I)=k$ if and only if $\mathcal I$ is $k$-maximal.
\end{proposition}

We now compute $\mathfrak{ad}(\mathcal I\times\mathcal J)$. 
To this end, we need to introduce some notation. 
For $A\subseteq\mathbb N\times\mathbb N$, recall that $A_{(n)} =\{m\in \N:(n,m)\in A\}$. Given an  ideal $\mathcal I$ on $\mathbb N$, we set $\varphi_\mathcal I(A)=\{n\in\mathbb N\,\colon\,A_{(n)}\not\in\mathcal I\}.$ We have the following facts:
\begin{enumerate}[(1)]
    \item If $A,B\subseteq\mathbb N^2$, then $\varphi_\mathcal I(A\cap B)\subseteq \varphi_\mathcal I(A)\cap\varphi_\mathcal I(B)$;
    
    \item  If $\{\mathcal J_i\,\colon\,i\in J\}$ is a family of ideals on $\mathbb N$ and $\mathcal J=\bigcap_{i\in J}\mathcal J_i$, then $\varphi_{\mathcal J}(A)=\bigcup_{i\in J}\varphi_{\mathcal J_i}(A)$;

    \item Let $\{\mathcal I_j\,\colon\,j\in I\}$ and $\{\mathcal J_i\,\colon\,i\in J\}$ be families of ideals on $\mathbb N$, with $J$ finite. If $\mathcal I=\bigcap_{j\in I}\mathcal I_j$
    and $\mathcal J=\bigcap_{i\in J}\mathcal J_i$, then:
    \begin{gather}
    \label{formulas fubini}
       \mathcal I\times\mathcal J=\bigcap_{j\in I}(\mathcal I_j\times\mathcal J),\quad \mbox{and}
       \quad \mathcal I\times\mathcal J=\bigcap_{i\in J}(\mathcal I\times\mathcal J_i).
    \end{gather}
\end{enumerate}
 Finally, it is not difficult to check that
if $\mathcal I$ and $\mathcal J$ are maximal ideals on $\mathbb N$, then $\mathcal I\times\mathcal J$ is maximal on $\mathbb N\times\mathbb N$.

\begin{theorem}\label{thm: ad(IxF)}
Let $\mathcal I$ and $\mathcal J$ be proper and admissible ideals on $\mathbb N$. Then:
\[
  \mathfrak{ad}(\mathcal I\times\mathcal J)=
  \begin{cases}
    k\cdot\mathfrak{ad}(\mathcal I),
      & \text{if }\mathfrak{ad}(\mathcal J)=k<\aleph_0,\\[2pt]
    \mathfrak c,
      & \text{if }\mathfrak{ad}(\mathcal J)\geq\aleph_0.
  \end{cases}
\]
\end{theorem}

\begin{proof}
    Recall that $\varphi_{\mathcal J}(A)=\{n\in\mathbb N:A_{(n)}\notin\mathcal J\}$ and that $A\in\mathcal I\times\mathcal J$ iff
    $\varphi_{\mathcal J}(A)\in\mathcal I$.

\begin{claim} 
    If $\mathcal A$ is $\mathcal I$-AD and $\mathcal B$ is $\mathcal J$-AD, then $\{E\times F:E\in\mathcal A,\ F\in\mathcal B\}$ is $(\mathcal I\times\mathcal J)$-AD.

    Indeed, notice that for each $E\times F$ with $E\in\mathcal A, F\in\mathcal B$, we have:
    \[ (E\times F)_{(n)} =
    \{m\in\mathbb N:(n,m)\in E\times F\} =
    \begin{cases}
        F,& n\in E\\ \emptyset,& n\notin E
    \end{cases},\]
    thus, $\varphi_\J(E\times F) = E \notin \I$ and $E\times F\in (\I\times \J)^+$.   

    Given $E\times F, E'\times F' \in \mathcal A\times \mathcal B$ distict, we need to check that $(E\cap E')\times(F\cap F')\in \I\times \J$.    
    Denoting $E\cap E' = I\in \I$ and $F\cap F' = J\in \J$, we have:
    \[
    (I\times J)_{(n)} 
        = \{ m\in\N: (n,m)\in I\times J\}
        = \begin{cases}
            J,& n\in I\\
            \emptyset,& n\notin I
        \end{cases}.    
    \]
    If $F\neq F'$, then $F\cap F' \in \J$ and $\varphi_\J(I\times J)=\emptyset \in \I$. If $F= F'$, then $\varphi_\J(I\times J)=E\cap E'\in \I$. The claim is proved.  
\end{claim}

\begin{claim}
    $\mathfrak{ad}(\mathcal I)\cdot\mathfrak{ad}(\mathcal J)
    = \sup_{\mathcal A,\mathcal B} |\mathcal A|\cdot|\mathcal B|$.

    Denote $\kappa = \mathfrak{ad}(\I)$ and $\lambda = \mathfrak{ad}(\J)$. Then $\sup_{\mathcal A,\mathcal B} |\mathcal A|\cdot|\mathcal B|\leq \kappa\cdot\lambda$. 
    
    If $\kappa$ and $\lambda$ are finite, the suprema on the right side are attained and equality is clear. Otherwise, we may suppose that $\mathfrak{ad}(\mathcal I)\cdot\mathfrak{ad}(\mathcal J) = \max\{\kappa,\lambda\} = \kappa$. 
    Fixing $\mathcal B = \{\N\}$ yields $\sup_{\mathcal A,\mathcal B} |\mathcal A|\cdot|\mathcal B| \geq 
    \sup_\mathcal A |\mathcal A|\cdot |\{\N\} = \kappa$.
\end{claim}

By the previous claims, if we take the suprema over $\mathcal A$ and $\mathcal B$, we obtain:
\[
\mathfrak{ad}(\mathcal I)\cdot\mathfrak{ad}(\mathcal J)
    = \sup_{\mathcal A,\mathcal B} |\mathcal A|\cdot|\mathcal B|
    \leq \mathfrak{ad}(\mathcal I\times\mathcal J).
\]

\bigskip

Suppose first that $\mathfrak{ad}(\mathcal J)=k<\aleph_0$.
By Proposition~\ref{caracterizacion ad(I) finite}, write $\mathcal J=\bigcap_{i=1}^k\mathcal J_i$, where each $\mathcal J_i$ is maximal and $\varphi_i=\varphi_{\mathcal J_i}$.

Let $\mathcal F$ be an $(\mathcal I\times\mathcal J)$-AD family.
For each $A\in\mathcal F$, let $i(A)$ be the least $j\in\{1,\ldots,k\}$ such that $\varphi_j(A)\notin\mathcal I$.
Such a $j$ exists because
$\varphi_{\mathcal J}(A)=\bigcup_{j=1}^k\varphi_j(A)\notin\mathcal I$.
For $i=1,\ldots,k$, set $\mathcal F_i=\{A\in\mathcal F:i(A)=i\}$.
Then $\mathcal F$ is the disjoint union of $\mathcal F_1,\ldots,\mathcal F_k$.

Since $\mathcal J_i^+$ is an ultrafilter, distinct $A,B\in\mathcal F_i$ satisfy
\[
  \varphi_i(A)\cap\varphi_i(B)
  =\varphi_i(A\cap B)
  \subseteq\varphi_{\mathcal J}(A\cap B)\in\mathcal I.
\]
Thus $\varphi_i$ maps $\mathcal F_i$ injectively onto an $\mathcal I$-AD family.
Hence $|\mathcal F_i|\leq\mathfrak{ad}(\mathcal I)$ for each $i$, and
$|\mathcal F|\leq k\cdot\mathfrak{ad}(\mathcal I)$.

Now suppose that $\mathfrak{ad}(\mathcal J)\geq\aleph_0$. For each
$n\in\mathbb N$, choose a $\mathcal J$-AD family $\{D_s:s\in 2^n\}$
of size $2^n$.
For each $x\in 2^\omega$, we define:
\[
  A_x=\bigcup_{n\in\mathbb N}\bigl(\{n\}\times D_{x\upharpoonright n}\bigr).
\]

It follows that $(A_x)_{(n)}=D_{x\upharpoonright n}$ for each $n\in\mathbb N$. Moreover, each of these sections is $\mathcal J$-positive and $\varphi_{\mathcal J}(A_x)=\mathbb N$.

Observe that, if $x,y\in 2^\omega$ are distinct, then:
\[
  \varphi_{\mathcal J}(A_x\cap A_y)
  =\{n\in\mathbb N:x\upharpoonright n=y\upharpoonright n\} \in \fin \subseteq \I.
\]

Thus $A_x$ form an $(\mathcal I\times\mathcal J)$-AD family of size $\mathfrak c$.
The reverse bound follows from $|\mathcal P(\mathbb N^2)|=\mathfrak c$.
\end{proof}

The construction with rectangles makes it natural to expect the product
formula $\mathfrak{ad}(\mathcal I\times\mathcal J)
=\mathfrak{ad}(\mathcal I)\cdot\mathfrak{ad}(\mathcal J)$.
The following example shows that this lower bound can be strict even when
$\mathfrak{ad}(\mathcal I)=1$ and $\mathfrak{ad}(\mathcal J)=\aleph_0$.

\begin{example}\label{exa: id changing factors}
Partition $\mathbb N$ into infinite sets $B_s$, indexed by $s\in 2^{<\omega}$.
For each $s$, choose a free ultrafilter $p_s$ on $\mathbb N$ containing
$B_s$, and let
\[
  \mathcal J_s=\mathcal P(\mathbb N)\setminus p_s,\qquad
  \mathcal J=\bigcap_{s\in 2^{<\omega}}\mathcal J_s.
\]
The blocks $B_s$ form a countably infinite $\mathcal J$-AD family.
Conversely, every $\mathcal J$-positive set belongs to some $p_s$,
and no $p_s$ can contain two distinct members of a $\mathcal J$-AD
family, since it would also contain their intersection. Hence
$\mathfrak{ad}(\mathcal J)=\aleph_0$.

Take any maximal ideal $\mathcal I$ on $\mathbb N$ containing all finite
sets. Its positive sets form an ultrafilter, so $\mathfrak{ad}(\mathcal I)=1$.
For $x\in 2^\omega$, let $(A_x)_{(n)}=B_{x\upharpoonright n}$. Every
section of $A_x$ is $\mathcal J$-positive. For $x\ne y$, the sections
of $A_x\cap A_y$ are empty at all but finitely many indices.
It follows that
\[
  \mathfrak{ad}(\mathcal I\times\mathcal J)
  =\mathfrak c>\aleph_0
  =\mathfrak{ad}(\mathcal I)\cdot\mathfrak{ad}(\mathcal J).
\]
Reversing the factors gives $\mathfrak{ad}(\mathcal J\times\mathcal I)
=\aleph_0$ by the finite case of Theorem~\ref{thm: ad(IxF)}.
\end{example}

\section{Acknowledgments}
The first and third authors were supported by VIE, projects 4248 and C-2026-1. The first author was also supported by the VIE-UIS Mobility Program, process 5644. 

\medskip
The first and second authors were supported by the São Paulo Research Foundation
(FAPESP), grant no. 2025/08683-7.

\medskip

The authors would like to thank the organizers and participants of the fourth edition of the Brazilian Workshop on Banach Spaces for the stimulating atmosphere and many fruitful discussions. Part of the final preparation of this manuscript was carried out outside the workshop sessions.
\medskip

We are also grateful to Professor Thomas Schlumprecht for his attention and helpful comments during the event. In particular, when discussing one of the results used in this paper, he pointed out that the literature often cites his Ph.D. thesis rather than the published article containing the corresponding theorem. When we asked him for the proper reference, however, he was unable to locate the article himself and ultimately advised us to cite the thesis. In view of this episode, we shall follow the tradition in functional analysis and continue citing his Ph.D. thesis.

\section*{Declaration of generative AI and AI-assisted technologies}

OpenAI Codex, powered by the GPT-6.1 Sol model, was used as a research and writing aid during the preparation of this manuscript. In particular, the model was used to discuss and explore some preliminary ideas and results developed by the authors over the course of the preceding year, occasionally providing useful insights or suggesting possible directions for further investigation. All mathematical arguments, proofs, and final formulations were independently checked, developed, and approved by the authors.


\begin{thebibliography}{10}
\expandafter\ifx\csname url\endcsname\relax
  \def\url#1{\texttt{#1}}\fi
\expandafter\ifx\csname urlprefix\endcsname\relax\def\urlprefix{URL }\fi
\expandafter\ifx\csname href\endcsname\relax
  \def\href#1#2{#2} \def\path#1{#1}\fi


\bibitem{AK}
F.~Albiac, N.~J. Kalton, \href{https://doi.org/10.1007/978-3-319-31557-7}{Topics in {B}anach space theory}, 2nd Edition, Vol. 233 of Graduate Texts in Mathematics, Springer, [Cham], 2016, with a foreword by Gilles Godefroy.
\newblock \href {https://doi.org/10.1007/978-3-319-31557-7} {\path{doi:10.1007/978-3-319-31557-7}}.
\newline\urlprefix\url{https://doi.org/10.1007/978-3-319-31557-7}

\bibitem{Aliprantis2006}
C.~D. Aliprantis, K.~C. Border, \href{http://dx.doi.org/10.1007/3-540-29587-9}{Infinite Dimensional Analysis: A Hitchhiker's Guide}, 3rd Edition, Springer-Verlag, 2006.
\newblock \href {https://doi.org/10.1007/3-540-29587-9} {\path{doi:10.1007/3-540-29587-9}}.
\newline\urlprefix\url{http://dx.doi.org/10.1007/3-540-29587-9}

\bibitem{argyros}
S.~A. Argyros, J.~F. Castillo, A.~S. Granero, M.~Jim{\'e}nez, J.~P. Moreno, Complementation and embeddings of {$c_0(I)$} in {B}anach spaces, Proceedings of the London Mathematical Society 85~(3) (2002) 742--768.

\bibitem{aviles-et-all}
A.~Avilés, F.~Cabello~S\'anchez, J.~M.~F. Castillo, M.~Gonz\'alez, Y.~Moreno, \href{https://doi.org/10.1007/978-3-319-14741-3}{Separably injective {B}anach spaces}, Vol. 2132 of Lecture Notes in Mathematics, Springer, [Cham], 2016.
\newblock \href {https://doi.org/10.1007/978-3-319-14741-3} {\path{doi:10.1007/978-3-319-14741-3}}.
\newline\urlprefix\url{https://doi.org/10.1007/978-3-319-14741-3}

\bibitem{balcerzak}
M.~Balcerzak, M.~Popławski, A.~Wachowicz, \href{https://doi.org/10.1016/j.topol.2017.09.014}{The {B}aire category of ideal convergent subseries and rearrangements}, Topology Appl. 231 (2017) 219--230.
\newblock \href {https://doi.org/10.1016/j.topol.2017.09.014} {\path{doi:10.1016/j.topol.2017.09.014}}.
\newline\urlprefix\url{https://doi.org/10.1016/j.topol.2017.09.014}

\bibitem{b-s-w}
A.~Bartoszewicz, S.~G\l\c{a}b, A.~Wachowicz, \href{http://dx.doi.org/10.1016/j.jmaa.2010.09.023}{Remarks on ideal boundedness, convergence and variation of sequences}, Journal of Mathematical Analysis and Applications 375~(2) (2011) 431--435.
\newblock \href {https://doi.org/10.1016/j.jmaa.2010.09.023} {\path{doi:10.1016/j.jmaa.2010.09.023}}.
\newline\urlprefix\url{http://dx.doi.org/10.1016/j.jmaa.2010.09.023}

\bibitem{sanchez2006yet}
F.~Cabello~S{\'a}nchez, Yet another proof of {S}obczyk's theorem, in: J.~M.~F. Castillo, W.~B. Johnson (Eds.), Methods in Banach Space Theory, Vol. 337 of London Mathematical Society Lecture Note Series, Cambridge University Press, 2006, pp. 133--138.
\newblock \href {https://doi.org/10.1017/CBO9780511721366.006} {\path{doi:10.1017/CBO9780511721366.006}}.

\bibitem{atob2000sobczyk}
F.~Cabello S{\'a}nchez, J.~M.~F. Castillo, D.~Yost, Sobczyk's theorems from {A} to {B}, Extracta Math 15~(2) (2000) 391--420.

\bibitem{cembranos}
P.~Cembranos, \href{https://doi.org/10.2307/2044800}{{$C(K,\,E)$}\ contains a complemented copy of {$c\sb{0}$}}, Proc. Amer. Math. Soc. 91~(4) (1984) 556--558.
\newblock \href {https://doi.org/10.2307/2044800} {\path{doi:10.2307/2044800}}.
\newline\urlprefix\url{https://doi.org/10.2307/2044800}

\bibitem{correa2013extensions}
C.~Correa, D.~V. Tausk, On extensions of $c_0$-valued operators, Journal of Mathematical Analysis and Applications 405~(2) (2013) 400--408.

\bibitem{correa2014compact}
C.~Correa, D.~V. Tausk, Compact lines and the {S}obczyk property, Journal of Functional Analysis 266~(9) (2014) 5765--5778.

\bibitem{correa2015c0}
C.~Correa, D.~V. Tausk, On the $c_0$-extension property for compact lines, Journal of Mathematical Analysis and Applications 428~(1) (2015) 184--193.

\bibitem{diestel-and-all}
J.~Diestel, J.~Fourie, J.~Swart, \href{https://doi.org/10.2989/16073600209486003}{The metric theory of tensor products ({G}rothendieck's r\'esum\'e{} revisited). {I}. {T}ensor norms}, Quaest. Math. 25~(1) (2002) 37--72.
\newblock \href {https://doi.org/10.2989/16073600209486003} {\path{doi:10.2989/16073600209486003}}.
\newline\urlprefix\url{https://doi.org/10.2989/16073600209486003}

\bibitem{fabian-et-al}
M.~Fabian, P.~Habala, P.~Hájek, V.~Montesinos, V.~Zizler, \href{http://dx.doi.org/10.1007/978-1-4419-7515-7}{{B}anach Space Theory: The Basis for Linear and Nonlinear Analysis}, CMS Books in Mathematics, Springer New York, 2011.
\newblock \href {https://doi.org/10.1007/978-1-4419-7515-7} {\path{doi:10.1007/978-1-4419-7515-7}}.
\newline\urlprefix\url{http://dx.doi.org/10.1007/978-1-4419-7515-7}

\bibitem{farkas-khomskii-vidnyanszky}
B.~Farkas, Y.~Khomskii, Z.~Vidnyánszky, \href{http://dx.doi.org/10.4064/fm429-7-2017}{Almost disjoint refinements and mixing reals}, Fundamenta Mathematicae 242~(1) (2018) 25--48.
\newblock \href {https://doi.org/10.4064/fm429-7-2017} {\path{doi:10.4064/fm429-7-2017}}.
\newline\urlprefix\url{http://dx.doi.org/10.4064/fm429-7-2017}

\bibitem{fast}
H.~Fast, \href{https://doi.org/10.4064/cm-2-3-4-241-244}{Sur la convergence statistique}, Colloq. Math. 2 (1951) 241--244.
\newblock \href {https://doi.org/10.4064/cm-2-3-4-241-244} {\path{doi:10.4064/cm-2-3-4-241-244}}.
\newline\urlprefix\url{https://doi.org/10.4064/cm-2-3-4-241-244}

\bibitem{filipow}
R.~Filip\'ow, N.~Mro\.zek, I.~Rec\l aw, P.~Szuca, \href{https://doi.org/10.2178/jsl/1185803621}{Ideal convergence of bounded sequences}, J. Symbolic Logic 72~(2) (2007) 501--512.
\newblock \href {https://doi.org/10.2178/jsl/1185803621} {\path{doi:10.2178/jsl/1185803621}}.
\newline\urlprefix\url{https://doi.org/10.2178/jsl/1185803621}

\bibitem{Fremlin}
D.~H. Fremlin, Measure Theory, Volume 3: Measure Algebras, Torres Fremlin, Colchester, UK, 2002.

\bibitem{Fremlin5}
D.~H.~Fremlin, \emph{Measure theory. Vol.~5. Set-theoretic measure theory.
Part~I}, corrected reprint of the 2008 original, Torres Fremlin,
Colchester, 2015, 329~pp. ISBN 978-0-9538129-5-0. MR3723040.

\bibitem{Garlingbetagamadual}
D.~J.~H. Garling, \href{https://doi.org/10.1017/s0305004100041992}{The {$\beta $}- and {$\gamma $}-duality of sequence spaces}, Proc. Cambridge Philos. Soc. 63 (1967) 963--981.
\newblock \href {https://doi.org/10.1017/s0305004100041992} {\path{doi:10.1017/s0305004100041992}}.
\newline\urlprefix\url{https://doi.org/10.1017/s0305004100041992}

\bibitem{GarlingTopologicalsequence}
D.~J.~H. Garling, \href{https://doi.org/10.1017/s0305004100042031}{On topological sequence spaces}, Proc. Cambridge Philos. Soc. 63 (1967) 997--1019.
\newblock \href {https://doi.org/10.1017/s0305004100042031} {\path{doi:10.1017/s0305004100042031}}.
\newline\urlprefix\url{https://doi.org/10.1017/s0305004100042031}

\bibitem{Grebik}
J.~Grebík, \href{http://dx.doi.org/10.14712/1213-7243.2015.279}{Ultrafilter extensions of asymptotic density}, Commentationes Mathematicae Universitatis Carolinae 60~(1) (2019) 25--37.
\newblock \href {https://doi.org/10.14712/1213-7243.2015.279} {\path{doi:10.14712/1213-7243.2015.279}}.
\newline\urlprefix\url{http://dx.doi.org/10.14712/1213-7243.2015.279}

\bibitem{hrusák2026grothendieckidealsellinfty}
M.~Hrušák, M.~A. Rincón-Villamizar, L.~S\'aenz, C.~Uzcátegui-Aylwin, \href{https://arxiv.org/abs/2607.25095}{Grothendieck ideals of $\ell_\infty$} (2026).
\newblock \href {http://arxiv.org/abs/2607.25095} {\path{arXiv:2607.25095}}.
\newline\urlprefix\url{https://arxiv.org/abs/2607.25095}

\bibitem{HrusakSaenz}
M.~Hru\v{s}\'ak, L.~S\'aenz, \href{http://dx.doi.org/10.4064/cm9701-11-2025}{Complemented ideals of {$\ell_\infty$}}, Colloquium Mathematicum 180~(2) (2026) 137--143.
\newblock \href {https://doi.org/10.4064/cm9701-11-2025} {\path{doi:10.4064/cm9701-11-2025}}.
\newline\urlprefix\url{http://dx.doi.org/10.4064/cm9701-11-2025}

\bibitem{kadetsetal2022}
V.~Kadets, D.~Seliutin, J.~Tryba, \href{http://dx.doi.org/10.1016/j.jmaa.2021.125955}{Conglomerated filters and statistical measures}, Journal of Mathematical Analysis and Applications 509~(1) (2022) 125955.
\newblock \href {https://doi.org/10.1016/j.jmaa.2021.125955} {\path{doi:10.1016/j.jmaa.2021.125955}}.
\newline\urlprefix\url{http://dx.doi.org/10.1016/j.jmaa.2021.125955}

\bibitem{kania}
T.~Kania, \href{https://doi.org/10.1007/s00025-018-0936-0}{A letter concerning {L}eonetti's paper `{C}ontinuous projections onto ideal convergent sequences'}, Results Math. 74~(1) (2019) Paper No. 12, 4.
\newblock \href {https://doi.org/10.1007/s00025-018-0936-0} {\path{doi:10.1007/s00025-018-0936-0}}.
\newline\urlprefix\url{https://doi.org/10.1007/s00025-018-0936-0}

\bibitem{k-m-s}
P.~Kostyrko, M.~Ma{\v c}aj, T.~{\v S}al\'at, M.~Sleziak, {$\mathscr I$}-convergence and extremal {$\mathscr I$}-limit points, Math. Slovaca 55~(4) (2005) 443--464.

\bibitem{k-s-w}
P.~Kostyrko, T.~{\v S}al\'at, W.~Wilczyński, {$\mathscr I$}-convergence, Real Anal. Exchange 26~(2) (2000/01) 669--685.

\bibitem{kothetopvec1}
G.~K\"othe, Topological vector spaces. {I}, Vol. Band 159 of Die Grundlehren der mathematischen Wissenschaften, Springer-Verlag New York, Inc., New York, 1969, translated from the German by D. J. H. Garling.


\bibitem{Leonetti2018}
P.~Leonetti, \href{https://doi.org/10.1007/s00025-018-0876-8}{Continuous projections onto ideal convergent sequences}, Results Math. 73~(3) (2018) Paper No. 114, 5.
\newblock \href {https://doi.org/10.1007/s00025-018-0876-8} {\path{doi:10.1007/s00025-018-0876-8}}.
\newline\urlprefix\url{https://doi.org/10.1007/s00025-018-0876-8}

\bibitem{leung-rabiger}
D.~Leung, F.~R\"abiger, \href{http://projecteuclid.org/euclid.ijm/1255988491}{Complemented copies of {$c_0$} in {$\ell^\infty$} sums of {B}anach spaces}, Illinois J. Math. 34~(1) (1990) 52--58.
\newline\urlprefix\url{http://projecteuclid.org/euclid.ijm/1255988491}

\bibitem{MagerlNamioka}
G.~M\"{a}gerl, I.~Namioka, \href{http://dx.doi.org/10.1007/BF01363900}{Intersection numbers and weak separability of spaces of measures}, Mathematische Annalen 249~(3) (1980) 273--279.
\newblock \href {https://doi.org/10.1007/bf01363900} {\path{doi:10.1007/bf01363900}}.
\newline\urlprefix\url{http://dx.doi.org/10.1007/BF01363900}

\bibitem{molto}
A.~Molt{\'o}, On a theorem of {S}obczyk, Bulletin of the Australian Mathematical Society 43~(1) (1991) 123--130.

\bibitem{patterson}
W.~M. Patterson, Complemented $c_0$-subspaces of a non-separable {$C(K)$}-space, Canadian Mathematical Bulletin 36~(3) (1993) 351--357.

\bibitem{rincon-uzcategui}
M.~A. Rincón-Villamizar, C.~Uzcátegui-Aylwin, \href{https://doi.org/10.1016/j.jmaa.2024.128271}{Banach spaces of {$\mathcal{I}$}-convergent sequences}, J. Math. Anal. Appl. 536~(2) (2024) Paper No. 128271, 19.
\newblock \href {https://doi.org/10.1016/j.jmaa.2024.128271} {\path{doi:10.1016/j.jmaa.2024.128271}}.
\newline\urlprefix\url{https://doi.org/10.1016/j.jmaa.2024.128271}

\bibitem{rincon-uzcateguiII}
M.~A. Rincón-Villamizar, C.~Uzcátegui-Aylwin, \href{https://arxiv.org/abs/2507.13866}{On the complementation of spaces of {$\mathcal I$}-null sequences} (2026).
\newblock \href {http://arxiv.org/abs/2507.13866} {\path{arXiv:2507.13866}}.

\bibitem{VT}
V.~d.~S. Ronchim, D.~V. Tausk, \href{https://doi.org/10.4064/sm211120-2-6}{Extension of {$c_0(I)$}-valued operators on spaces of continuous functions on compact lines}, Studia Math. 268~(3) (2023) 259--289.
\newblock \href {https://doi.org/10.4064/sm211120-2-6} {\path{doi:10.4064/sm211120-2-6}}.
\newline\urlprefix\url{https://doi.org/10.4064/sm211120-2-6}

\bibitem{rosenthal}
H.~P. Rosenthal, \href{http://dx.doi.org/10.1007/BF02394572}{On injective {B}anach spaces and the spaces {$L^\infty(\mu)$} for finite measures {$\mu$}}, Acta Mathematica 124 (1970) 205--248.
\newblock \href {https://doi.org/10.1007/bf02394572} {\path{doi:10.1007/bf02394572}}.
\newline\urlprefix\url{http://dx.doi.org/10.1007/BF02394572}

\bibitem{Sargent1966}
W.~L.~C. Sargent, \href{http://dx.doi.org/10.1112/jlms/s1-41.1.79}{On compact matrix transformations between sectionally bounded {BK}-spaces}, Journal of the London Mathematical Society s1-41~(1) (1966) 79--87.
\newblock \href {https://doi.org/10.1112/jlms/s1-41.1.79} {\path{doi:10.1112/jlms/s1-41.1.79}}.
\newline\urlprefix\url{http://dx.doi.org/10.1112/jlms/s1-41.1.79}

\bibitem{schaeffer}
H.~H. Schaefer, Banach lattices and positive operators, Vol. Band 215 of Die Grundlehren der mathematischen Wissenschaften, Springer-Verlag, New York-Heidelberg, 1974.

\bibitem{schlumprecht1987limitierte}
T.~Schlumprecht, Limitierte {M}engen in {B}anachr{\"a}umen, Dissertation, Ludwig-Maximilians-Universit{\"a}t M{\"u}nchen, Munich, Germany (1987).

\bibitem{WilanskySummability}
A.~Wilansky, Summability through functional analysis, Vol.~85 of North-Holland Mathematics Studies, North-Holland Publishing Co., Amsterdam, 1984, Notas de Matem\'atica, 91. [Mathematical Notes].
\end{thebibliography}
\end{document}